\documentclass[pdf]{article}

\usepackage{pdfpages}
\usepackage{hyperref}
\usepackage{xcolor}
\usepackage[utf8]{inputenc}
\usepackage{amsmath}
\usepackage{amsfonts}
\usepackage{amssymb}
\usepackage{amsthm}
\usepackage{marginnote}
\usepackage{enumerate}
\usepackage{graphicx}
\usepackage[normalem]{ulem}
\usepackage{fancyhdr}
\newcommand{\defstyle}[1]{\textbf{#1}}
\newcommand{\myprob}[1]{\mathbb P \left[ #1 \right]}

\newcommand{\probCond}[2]{\mathbb P \left[ #1 \left| #2 \right. \right]}

\newcommand{\omid}[1]{\mathbb E \left[ #1 \right]}

\newcommand{\bs}[1]{\boldsymbol{#1}}
\newcommand{\card}[1]{\# #1}
\newcommand{\ave}{\mathrm{ave}}

\newcommand{\del}[1]{}
\newcommand{\mar}[1]{}
\newcommand{\unwritten}[1]{}

\newcommand{\restrict}[2]{{% we make the whole thing an ordinary symbol
		\left.\kern-\nulldelimiterspace % automatically resize the bar with \right
		#1 % the function
		\vphantom{\big|} % pretend it's a little taller at normal size
		\right|_{#2} % this is the delimiter
}}

\theoremstyle{theorem}
\newtheorem{theorem}{Theorem}[section]
\newtheorem{lemma}[theorem]{Lemma}
\newtheorem{proposition}[theorem]{Proposition}

\newtheorem{conjecture}[theorem]{Conjecture}

\theoremstyle{definition}
\newtheorem{definition}[theorem]{Definition}

\newtheorem{example}[theorem]{Example}
\theoremstyle{definition}
\newtheorem{remark}[theorem]{Remark}

\theoremstyle{theorem}

\numberwithin{equation}{section}

\makeatletter
\let\orgdescriptionlabel\descriptionlabel
\renewcommand*{\descriptionlabel}[1]{%
	\let\orglabel\label
	\let\label\@gobble
	\phantomsection
	\edef\@currentlabel{#1}%
	\let\label\orglabel
	\orgdescriptionlabel{#1}%
}

\makeatother

\begin{document}

\title{Counter Examples to Invariant Circle Packing}
%\title{Unimodular Planar Graphs with No (Point-) Stationary Circle Packing}

\author{Ali Khezeli\footnote{Department of Applied Mathematics, Faculty of Mathematical Sciences, Tarbiat Modares University, P.O. Box 4115-134, Tehran, Iran, khezeli@modares.ac.ir}\ \footnote{School of Mathematics, Institute for Research in Fundamental Sciences (IPM), P.O. Box: 19395-5746, Tehran, Iran, alikhezeli@ipm.ir}}

\maketitle

\begin{abstract}
	In this work, a unimodular random planar triangulation is constructed that has no invariant circle packing. This disputes a problem asked in~\cite{BeTi19}. A natural weaker problem is the existence of point-stationary circle packings for a graph, which are circle packings that satisfy a certain mass transport principle. It is shown that the answer to this weaker problem is also false. Two examples are provided with two different approaches: Using indistinguishability and finite approximations.
	
	%\textbf{Keywords.} Later
\end{abstract}

%\setcounter{tocdepth}{2}
%\tableofcontents

\section{Introduction}
A well known theorem of Koebe, Andreev and Thurston~\cite{bookKo36,bookTh79} states that every finite simple planar graph can be represented by a \defstyle{circle packing}; i.e., one can correspond a circle to every vertex such that the circles have disjoint interiors and two circles are tangent if and only if their corresponding vertices are adjacent. In addition, if the graph is a triangulation (i.e., every face has 3 edges), then such a circle packing is unique up to Mobius transformations. Circle packings have attracted a lot of attention, especially because of their connections to conformal maps, hyperbolic manifolds and random walks.
%a conjecture by Thurston (later: ref) in ??? that circle packings can be used to approximate conformal maps (the conjecture has been proved in ???). They also have been useful for the study of hyperbolic manifolds (later: ref). Later: and random walks

The existence and uniqueness of circle packings have been extended to infinite graphs by He and Schramm~\cite{HeSc93,HeSc95}. In particular, they proved that every infinite planar triangulation with one {end} has a locally finite circle packing in either the plane or the hyperbolic plane (which can be represented by the unit disk), but not both. In the first case, the graph is called \textit{CP-parabolic} and in the second case, it is called \textit{CP-hyperbolic} . They also proved that such a circle packing is unique up to similarities of the plane (resp. isometries of the hyperbolic plane). 
There is also a rich theory that connects the geometry of the circle packing to the behavior of the simple random walk on the graph (see the discussion in~\cite{AnHuNaRa16}), assuming that the degrees of the vertices are bounded.
%There has been other influential works that use circle packings to study random walks on planar graphs, assuming that the vertices have bounded degrees (see ).
For instance, \cite{HeSc95} proves that the type of the circle packing is determined by recurrence or transience of the random walk. It should be noted that the assumption of bounded degrees is crucial for having general results. %Otherwise, for instance, one can construct special examples by adding circles to a given circle packing such that the properties of the random walk is changed.

Some models of random planar graphs have been of great interest recently; e.g., the UIPT (uniform infinite planar triangulation) \cite{AnSc03}. In these models, since the graphs mostly have unbounded degrees, many of the general results about circle packings cannot be  applied. For instance, the main goal of~\cite{GuNa13} is to prove the recurrence of UIPT.
More general than specific examples, \cite{AnHuNaRa18} proved that many of the general results about circle packings of bounded-degree triangulations can be generalized to all \textit{unimodular} random planar triangulations.
%More general than specific examples, ??? studies circle packings of general \textit{unimodular} triangulations. 
The concept of unimodularity of random (rooted) graphs, introduced in~\cite{processes}, can be thought of begin \textit{statistically homogeneous} and is defined by the \textit{mass transport principle} (see Subsection~\ref{subsec:sCP}). %Using this assumption, ??? proved that many of the general results about circle packings of bounded-degree triangulations can be generalized to all unimodular planar triangulations.
This notion is connected to stationary point processes as follows: Roughly speaking, by constructing a graph on a stationary point process without looking at where the origin is (i.e., in a translation-invariant manner), a unimodular graph is obtained. More precisely, one should condition on the event that the origin is included in the point process (this gives the \textit{Palm version} of the point process), and then take the origin as the root of the graph.
More generally, the same holds for \textit{point-stationary} point processes, which are point processes that contain the origin and satisfy a certain mass transport principle.

%By the mass transport principle, unimodular graphs are connected to stationary point processes as follows: Roughly speaking, by constructing a graph on a point-process in a translation-invariant manner, a unimodular graph is obtained. More precisely, one should condition on the event that the origin is included in the point process (this gives the \textit{Palm version} of the point process), and then take the origin as the root of the graph. Conversely, \cite{BeTi19} show that every unimodular triangulation can be represented by a stationary point process either in the Euclidean plane or in the hyperbolic plane. Later: problem of sCP.

Conversely, given a unimodular random planar graph, can it always be embedded in the plane (or the hyperbolic plane) such that the distribution of the embedded graph is invariant under all isometries (called an \textit{invariant embedding} in~\cite{BeTi19})? Under the condition of finite expected degree, the answer is yes and is proved in~\cite{BeTi19}. In addition, \cite{BeTi19} asks the following natural question: Does the graph have an \textit{invariant circle packing}? In other words, can one choose a version of the circle packing of the random graph such that the distribution of the circle packing is invariant under isometries (note that the graph does not have a unique circle packing since one can apply an isometry or a similarity)? \cite{BeTi19} shows that the answer is yes for CP-hyperbolic triangulations. One of the main ingredients of the proof is that the radii of the circles in the circle packing are determined by the graph. For CP-parabolic graphs, the radii of the circles are not determined since one can scale the set of circles arbitrarily. This freedom is an obstacle for the arguments to work in the CP-parabolic case, and hence, the question has been open for CP-parabolic graphs. In this paper, we dispute this problem as follows:

\begin{theorem}
	\label{thm:scp}
	There exists a unimodular triangulation with bounded degrees that is CP-parabolic but does not have any invariant circle packing.
\end{theorem}

A counterexample to prove this theorem will be provided in Section~\ref{sec:scpCounter}. Perhaps surprisingly, the proof does not use the freedom in choosing a scale at all! In fact, it will be shown that choosing a suitable scale, if possible, can only give a \textit{point-stationary} circle packing (Lemma~\ref{lem:equivCP}). To convert it into a stationary circle packing, another condition is required (that a certain Voronoi cell has finite expected area). So it is natural to weaken the problem asked in~\cite{BeTi19} as follows: Does every unimodular planar graph have a point-stationary circle packing? We will show that the answer to this problem is also negative, but with a much harder proof:

\begin{theorem}
	\label{thm:pscp}
	There exists a unimodular triangulation with bounded degrees that is CP-parabolic but does not have any point-stationary circle packing.
\end{theorem}

This theorem is proved in Section~\ref{sec:pscpCounter} by using the fact that certain sets in the provided example are indistinguishable (Lemma~\ref{lem:foilselection}).
Note that this theorem directly implies Theorem~\ref{thm:scp}, but the latter is proved separately since its proof is much simpler and is the basis of the proof of Theorem~\ref{thm:pscp}.

For a given unimodular planar graph, what criteria ensures the existence of a point-stationary circle packing? In section~\ref{sec:anotherapproach}, we investigate this problem using approximations of the graph (and its circle packing) by finite graphs (such approximations are always possible in theory, as proved in~\cite{AnHuNaRa18}). It will be shown that in a finite approximation of the circle packing, the empirical distribution of the radii gives information on the existence of a point-stationary circle packing. Using this approach, another example is proved for Theorem~\ref{thm:pscp}.

The paper is structures as follows. Section~\ref{sec:definitions} provides basic definitions and defines (point-) stationary circle packings. Sections~\ref{sec:scpCounter} and~\ref{sec:pscpCounter} prove Theorems~\ref{thm:scp} and~\ref{thm:pscp} respectively. The approach by finite approximations is provided in Section~\ref{sec:anotherapproach}.

\section{Definitions}
\label{sec:definitions}

This section provides basic definitions regarding circle packings and (point-) stationary circle packings.

\subsection{Circle Packings}

A \defstyle{circle packing}, abbreviated by CP, is a collection $P=\{C_v\}$ of circles in the plane such that the interiors of the circles are disjoint. The \defstyle{nerve} or \defstyle{tangency graph} of $P$ is the graph with vertex set $P$ obtained by connecting two vertices by an edge if and only if their corresponding circles are tangent. This graph can be embedded in the plane by putting the vertices in the centers of corresponding circles and connecting every pair of adjacent vertices by a straight segment. So the nerve is a planar simple graph.

Conversely, let $G$ be a planar simple graph. A \defstyle{circle packing of $G$} is a map $C$ that assigns to each vertex $v\in V(G)$ a circle $C_v$ in $\mathbb R^2$ such that $P:=\{C_v\}_{v\in V(G)}$ is a CP with nerve $G$. If $G$ is a \textit{plane graph} (i.e., already embedded in the plane), one can require that the natural embedding of the nerve of $P$ has the same combinatorial structure as $G$; i.e., can be obtained from $G$ by a homeomorphism of the plane. A similar property can be required if $G$ is equipped only with a \textit{combinatorial embedding} (in other words, it is a \textit{map}); i.e., equipped with a cyclic order on the set of neighbors of any vertex or equipped with a set of subsets which represent the {faces}.
Note that the mapping from $V(G)$ to the set of circles is distinguished. For instance, there are many CPs of the graph $\mathbb Z^2$ which consist of the circles of radii $\frac 1 2$ at integer points of the plane. 

It is proved that every locally finite simple plane graph $G$ has a CP~\cite{bookKo36,bookTh79,HeSc93,HeSc95}. Such an embedding is not unique if $G$ has faces with more than 3 vertices. However, if $G$ is a triangulation of the plane and has one end (i.e., by removing finitely many vertices, exactly one infinite connected component can appear), then there is a uniqueness result as follows~\cite{HeSc93,HeSc95}. 
Let $P$ be a circle packing of $G$.
The \defstyle{carrier} of $P$ is the union of the circles together with their interiors and the regions which correspond to the faces of $G$ and are bounded by three adjacent circles (note that since $P$ respects the combinatorial structure of the graph, each face corresponds to a bounded region in the plane). Since $P$ has one end, it can be seen that the carrier is a simply connected open subset of $\mathbb R^2$. It is proved that  exactly one of the following two cases happen: (1) The carrier of every circle packing of $G$ is $\mathbb R^2$. $G$ is called \defstyle{CP-parabolic} in this case. (2) There exists a circle packing of $G$ such that its carrier is the unit disk. $G$ is called \defstyle{CP-hyperbolic} in this case.
%The boundary of the carrier of $P$ in the sphere $\mathbb R^2\cup\{\infty\}$ is a closed curve. In this case, $G$ is called CP-hyperbolic and $P$ can be chosen such that its carrier is the unit disk. In this case, .
In this case, by regarding the unit disk as the hyperbolic plane (the Poinc\'are disk model) one obtains a CP of $G$ in the hyperbolic plane. In the first (resp. second) case, the circle packing is unique up to M\"obius transformations that fix the Euclidean (resp. hyperbolic) plane; i.e., up to isometries and similarities of the plane (resp. up to isometries of the hyperbolic plane).

%In this case, the CPs of $G$ are unique up to isometries of the plane and similarities.
%
%If so, $P$ is unique up to M\"obius transformations that fix the unit disk (which are isometries of the hyperbolic plane in the Poinc\'are disk model).

In this paper, by a \defstyle{triangulation} we always mean a locally finite simple plane graph that has one end and all faces have three edges. 
%then the circle packing of $G$ is unique up to M\"obius transformations (later: ref). More precisely, if $\{C_v\}_{v\in V(G)}$ and $\{C'_v\}_{v\in V(G)}$ are two circle packings of $G$, then there exists a M\"obius transformation $T$ such that $\forall v\in V(G): T(C_v)=C'_v$. In this paper, by a \defstyle{triangulation} we always mean such a graph; i.e., a locally finite simple plane graph that has one end and all faces have three edges. 
The above uniqueness readily implies the following result which is well known.

\begin{lemma}
	\label{lem:automorphism}
	If $G$ is a one-ended triangulation, $\tau$ is an automorphism of $G$ and $C=\{C_v\}_{v\in V(G)}$ is a circle packing of $G$ with carrier $\mathbb R^2$ (resp. the unit disc), then there exists an isometry $T$  of the plane (resp. the hyperbolic plane) that permutes the circles and represents $\tau$; i.e., $\forall v: T(C_v)=C_{\tau(v)}$.
\end{lemma}
\begin{proof}
	Note that $C'_v:=C_{\tau(v)}$ is another CP of $G$. By uniqueness, there exists a similarity $T$ of the plane such that $T(C_v)=C'_v=C_{\tau(v)}$. So $T$ permutes the circles. In the CP-hyperbolic case, $T$ is an isometry and the claim is proved. In the other case, assume $T$ is not an isometry, and hence, it has a fixed point. One can deduce that this fixed point is a limit point of the CP, which is a contradiction. So the claim is proved.
\end{proof}

\subsection{Stationary Circle Packings}
\label{subsec:sCP}

An \defstyle{invariant (or stationary) circle packing} in the plane (resp. in the hyperbolic plane) is a random CP in the plane such that its distribution is invariant under all translations of the plane (resp. isometries of the hyperbolic plane). 
%(here, we always assume that the tangency graph is a triangulation of the whole plane). 
In other words, the (random) set of centers of the circles is a \textit{stationary point process}, and moreover, by letting the mark of each center be the corresponding radius, one gets a \textit{stationary marked point process}. 

Consider the nerve of a stationary CP. Not all random planar graphs appear this way. A necessary condition is \textit{unimodularity}, described below, which is a counterpart of stationarity in the context of random graphs. Conversely, whether unimodularity is a sufficient condition or not, has been an open problem~\cite{BeTi19} and it disputed in this paper.

A \defstyle{rooted graph} is a pair $(G,o)$, where $o$ is a distinguished vertex of $G$. We always assume that the graph is connected and locally finite. The symbol $[G,o]$ shows the corresponding equivalence class, where two rooted graphs are equivalent if they are isomorphic. Let $\mathcal G_*$ be the set of these equivalence classes. It is known that $\mathcal G_*$, under the \textit{Benjamini-Schramm topology}, is a Polish space (see e.g., \cite{processes}). A \defstyle{random rooted graph} is a random object in $\mathcal G_*$ and is denoted by bold symbols like $[\bs G, \bs o]$. A \defstyle{unimodular graph}~\cite{processes} is a random rooted graph $[\bs G, \bs o]$ that satisfies the \textit{mass transport principle}: 
\begin{equation}
	\label{eq:mtp}
	\forall g: \omid{\sum_{v\in V(\bs G)} g(\bs G, \bs o, v)} = \omid{\sum_{v\in V(\bs G)} g(\bs G, v, \bs o)},
\end{equation}
where the function $g$ should be nonnegative, invariant under isomorphisms and measurable. One can also allow that every vertex or pair of vertices has some mark, which leads to the definition of \defstyle{unimodular marked graphs} (or networks) similarly. In particular, one can use marks to define \defstyle{unimodular planar graphs} (see Example~9.6 of~\cite{processes}).
%(later: mention the distribution of non-rooted graphs)

Assume that $\bs P$ is a stationary CP. One can choose a \textit{typical} root for the nerve of $\bs P$ formalized as follows: Let $\bs P_0$ be the CP obtained by conditioning $\bs P$ to have a circle centered at the origin (this is called the \textit{Palm version} of $\bs P$ and is defined if the set of centers of the circles has finite \textit{intensity}, see e.g., \cite{ThLa09}). Let $\bs G$ be the nerve of $\bs P_0$ and $\bs o$ be the origin. It can be shown that $[\bs G, \bs o]$ is unimodular. Conversely, if $[\bs G, \bs o]$ is a unimodular planar graph, then an \defstyle{invariant (or stationary) circle packing of $[\bs G, \bs o]$} is a stationary circle packing $\bs P$ such that the nerve of $\bs P_0$ has the same distribution as $[\bs G, \bs o]$.

Let $[\bs G, \bs o]$ be a unimodular planar graph. It is proved in~\cite{BeTi19} that if $\bs G$ is CP-hyperbolic a.s. and $\omid{\mathrm{deg}(\bs o)}<\infty$, then $[\bs G, \bs o]$ has a stationary CP in the hyperbolic plane. \cite{BeTi19} raised the problem that whether every CP-parabolic unimodular planar graph has a stationary CP in $\mathbb R^2$ or not. Here, it will be shown that the answer to this problem is negative (see Remark~\ref{rem:obstacle} for the reasons that the arguments in~\cite{BeTi19} do not work in this case).

%later:  say that the first counterexample has nothing to do with the freedom in the choice of the scale.).

%\begin{theorem}
%	\label{thm:scp}
%	\mar{be careful about bounded degree.}
%	There exists a unimodular triangulation with bounded degrees that is CP-parabolic but does not have any invariant circle packing.
%\end{theorem}
%
%%This theorem disputes a conjecture ... (later: ref). 
%Later: It will be proved in (later: cross ref).

\begin{remark}
	It is shown in~\cite{AnHuNaRa16} that for unimodular planar graphs, being \textit{invariantly non-amenable} is equivalent to being (with positive probability) CP-hyperbolic (see~\cite{AnHuNaRa16}).
\end{remark}

\begin{remark}
	\label{rem:obstacle}
	Let $[\bs G, \bs o]$ be a unimodular planar graph that is CP-parabolic. There are two reasons that might prevent it from having a stationary CP. Note that in every realization, there exists a CP in $\mathbb R^2$, but it is not unique since one can scale it arbitrarily. This freedom is a disadvantage since, to get a stationary CP, it is necessary to choose one of these infinitely many CPs in every realization of the graph without looking at the root (indeed, in the examples of this  paper, it is not possible to do this in a measurable way). %the results of this paper show that this is not necessarily possible). 
	Even finding a suitable scale (in a measurable way) is not enough to obtain a stationary CP (e.g., Example~\ref{ex:canopy}). In fact, it leads to a \textit{point-stationary} CP, describe in the next subsection (by Lemma~\ref{lem:equivCP}). 
	To convert it into a stationary CP, it is required that a certain Voronoi cell has finite expected are. This will be explained in Remark~\ref{rem:voronoi}.
	%The next necessary step would be to consider the Voronoi cell $C$ (of the set of centers) corresponding to the root, to bias the probability measure by the area of $C$, and then to shift a uniformly-at-random point of $C$ to the origin by a translation (this is similar to the proof of~\cite{BeTi19} for the CP-hyperbolic case). This requires that $C$ has finite expected area, which might not be the case even if $\omid{\mathrm{deg}(\bs o)}<\infty$ (later: a counterexample).
	%
	%However, if $[\bs G, \bs o]$ is CP-parabolic, then there are two reasons that the above arguments do not work. One is that $\bs G$ does not have a unique circle packing up to isometries of the plane (but only up to similarities). Hence, the radii of the circles are not uniquely determined and only the ratio of the radii are determined. In general, there might be no way to choose a suitable scale without looking at the root (later: explain more and cross ref). Even if one can select a global scale, the second issue is that the area of the Voronoi cell might have infinite expected are.
\end{remark}

\subsection{Point-Stationary Circle Packings}
\label{subsec:psCP}
According to Theorem~\ref{thm:scp} and Remark~\ref{rem:obstacle}, it is natural to ask whether every CP-parabolic unimodular planar graph has a \textit{point-stationary CP} or not. In this subsection, point-stationarity is defined. It will be shown in Section~\ref{sec:pscpCounter} that the answer to this question is also negative.

%As mentioned in Remark~\ref{rem:hyperbolicproof}, for every CP-parabolic planar graph, there is a freedom in choosing a scale for drawing a circle packing. This freedom makes a problem in finding a stationary circle packing for a given CP-parabolic unimodular graph. But the counterexample that is used to prove Theorem~\ref{thm:scp} has nothing to do with this problem. In fact, this problem is related to \textit{point-stationarity}, described below. A natural problem after Theorem~\ref{thm:scp} is whether every CP-parabolic unimodular planar graph has a point-stationary circle packing or not. We show that the answer to this question is also negative.

The fundamental equation that relates unimodularity to stationary point processes is the mass transport principle~\eqref{eq:mtp}. A similar formula holds for (Palm versions of) stationary point processes, but it also holds for \textit{point-stationary} point processes~\cite{ThLa09}. A (marked) point process $\Phi$ in the plane is \defstyle{point-stationary} if $0\in \Phi$ a.s. and %for every measurable function $g:...$ (later), one has
\[
	\forall g: \omid{\sum_{x\in \Phi} g(\Phi,0,x)} = \omid{\sum_{x\in \Phi} g(\Phi,x,0)}.
\] 
In this equation, $g$ is an arbitrary function that assigns a nonnegative number $g(\varphi,x,y)$ to every tuple $(\varphi,x,y)$, where $\varphi$ is a (marked) discrete subset of the plane and $x,y\in \varphi$, such that $g$ is invariant under translations and is measurable (to be suitably defined). In particular, this defines \defstyle{point-stationary circle packings}. Note that in such packings, there is always a circle centered at the origin. 
By the arguments in Subsection~\ref{subsec:sCP}, the Palm version of every stationary circle pacing is point-stationary. It can be seen that the nerve of every point-stationary CP, rooted at the origin, is a unimodular graph. Conversely, given a unimodular planar graph $[\bs G, \bs o]$, every point-stationary CP of $[\bs G, \bs o]$ is a random assignment of circles to the vertices of $\bs G$ such that the resulting CP is point-stationary (see \textit{equivariant circle packings} discussed below).

\begin{example}
	\label{ex:canopy}
	%	\begin{figure}[t]
	%		\centering
	%		\includegraphics[width=.5\textwidth]{images/canopy.jpg}
	%		\label{fig:canopy}
	%		\caption{The circle packing in Example~\ref{ex:canopy}}
	%	\end{figure}
	The following is a point-stationary CP.
	%The canopy tree (later: define it) has a point-stationary circle packing as follows. 
	Let $y_0<y_1<\cdots$ be a sequence which will be determined later. Let $U_0,U_1,\ldots$ be i.i.d. random numbers chosen uniformly in $\{\pm 1\}$ and $a_n:=\sum_{i=0}^{n-1}2^iU_i$. For every $n\geq 0$ and $m\in\mathbb Z$, put a ball of radius $(2^{n-1}-0.01)$ centered at $(\frac 12 a_n + m 2^n, y_n)$. It can be seen that the sequence $y_n$ can be uniquely chosen such that each circle at level $y_n$ is externally tangent to precisely two circles at level $y_{n-1}$ (if $n\geq 1$) and one circle at level $y_{n+1}$. Finally, let $X$ be a random point such that $X=(\frac 1 2  a_n, y_n)$ with probability $2^{-n-1}$ (given the sequence $(U_i)_i$). By translating these set of circles by vector $-X$, one obtains a point-stationary CP. The tangency graph is the \textit{canopy tree}.
	\\
	It is easy to see that this CP is not the Palm version of any stationary circle packing. % (equivalently, by Remark~\ref{rem:voronoi}, the Voronoi cell of the origin has infinite expected area). 
	More generally, we guess that the canopy tree does not have any stationary CP. % (see also the example in Subsection~\ref{subsec:sCP}).
\end{example}

\begin{remark}
	\label{rem:voronoi}
	Under the following condition, one can convert a point-stationary CP to a stationary CP. Consider the Voronoi tessellation of the set of centers and let $\bs C$ be cell of the origin. If $\bs C$ has finite expected area, then one can obtain a stationary CP as follows: Bias the probability measure by the area of $\bs C$ and then, move a random point of $\bs C$, chosen uniformly, to the origin by a translation. 
	This is similar to the proof of~\cite{BeTi19} for the CP-hyperbolic case
	%See~\cite{BeTi19} for the formal proof 
	and is a special case of the \textit{inversion formula} (see e.g., \cite{ThLa09}).
\end{remark}

Let $[\bs G, \bs o]$ be a unimodular planar graph. Assume that one assigns a number $\bs r(v)$ to every vertex $v\in V(\bs G)$ and a number $\bs d(u,v)$ to every pair $(u,v)$ of vertices, possibly using extra randomness, such that:\footnote{These conditions can also be formulated by \textit{unimodular embeddings} defined in~\cite{BeTi19}.}
\begin{enumerate}[(i)]
	\item Unimodularity is preserved by adding the marks (see Remark~\ref{rem:equivprocess}).
	\item In almost every realization, there exists a CP of $[\bs G, \bs o]$ such that $\bs r(v)$ is the radius of the circle corresponding to $v$ and $\bs d(u,v)$ is the distance of the centers of the corresponding circles (such a CP is not uniquely determined).
\end{enumerate}

Call the pair $\bs r$ and $\bs d$ an \defstyle{equivariant circle packing} of $[\bs G, \bs o]$. Note that it is not an actual CP since it determines a CP only up to isometries (and does this uniquely).
Note also that if $\bs G$ is a triangulation, then every realization of $\bs G$ has a unique CP up to similarities and there are infinitely many scalings of such a CP. In this case, heuristically, if one chooses one of these infinitely many CPs without looking at the root (in each realization), then an equivariant circle packing is obtained (see Remark~\ref{rem:equivprocess}).
%$\bs r$ uniquely determines $\bs d$, and also, the ratio $\bs r(u)/\bs r(v)$ is uniquely determined by the graph. In this case, 

\begin{remark}
	\label{rem:equivprocess}
	Roughly speaking, assigning extra marks to a unimodular graph (possibly using extra randomness) preserves unimodularity if and only if the assignment is done without looking at the root and also depends only on the isomorphism class of the graph (in a measurable way). See \textit{equivariant process} in~\cite{I} for a formal statement.
\end{remark}

\begin{lemma}
	\label{lem:equivCP}
	If $[\bs G, \bs o]$ has an equivariant circle packing, then it has a point-stationary circle packing.
	%If the random marks $\bs r$ and $\bs d$ exist as above, then there exists a point-stationary circle packing of $[\bs G, \bs o]$.
\end{lemma}
\begin{proof}
	%If so, one can obtain a point-stationary circle packing of $[\bs G, \bs o]$ as follows: 
	In every realization, choose an instance of the equivariant CP such that the circle corresponding to $\bs o$ is centered at the origin. Then, apply a random isometry that fixes the origin (using the uniform measure on rotations and the uniform measure on reflections). The result does not depend on the chosen instance since the equivariant CP uniquely determines a CP up to isometries. It can be seen that this CP is point-stationary. Note that there might be other point-stationary CPs of $[\bs G, \bs o]$ as well (e.g., the standard hexagonal CP is point-stationary and there is no need to apply a random isometry).
\end{proof}

\section{An Example With No Stationary CP}
\label{sec:scpCounter}
In this section, an example will be constructed to prove Theorem~\ref{thm:scp}. The proof is based on the following easy result.

\begin{theorem}
	\label{thm:scp2}
	Assume $[\bs G, \bs o]$ is a unimodular triangulation that has a unique automorphism other than the identity almost surely. Then $[\bs G, \bs o]$ has no stationary circle packing.
\end{theorem}

%It is easy to construct examples satisfying the above theorem, which is done in ??? below. This gives a proof of Theorem~\ref{thm:scp}.

\begin{proof}%[Proof of Theorem~\ref{thm:scp2}]
	Assume $\bs G$ is CP-parabolic for simplicity (the  general case is similar).
	Let $\bs P$ be a stationary CP of $[\bs G, \bs o]$. By Lemma~\ref{lem:automorphism}, there exists a unique isometry $\bs T$ of the plane (other than the identity and depending on $\bs P$) that permutes the circles in $\bs P$. So $\bs T\circ\bs T$ is the identity, and hence, $\bs T$ is either a reflection through a point, namely $\bs x$, or a reflection through a line, namely $\bs l$. It can be seen that $\bs T$ is a measurable function of $\bs P$. Therefore, in the first case, $\bs x$ is a random point in the plane such that its distribution is invariant under the translations, which is impossible. In the second case, $\bs l$ is a random line such that its distribution is translation-invariant, which is again impossible. So the claim is proved.
\end{proof}

Note \mar{Later: If I prove that there is no way to distinguish an automorphism, make this a proposition.} that the above result holds both in the Euclidean plane and the hyperbolic plane. Therefore, Theorem~3.1 of~\cite{BeTi19} directly implies that no nonamenable unimodular planar graph with finite expected degree can have a unique automorphism (other than the identity).

%Later: Prove in both hyperbolic case and Euclidean case. Say that in hyperbolic case, it is impossible to have such a graph!

%\begin{corollary}
%	(later: say directly: if nonamenable, then there is no way to select an automorphism (correct?))
%	Under the assumptions of Theorem~\ref{thm:scp2}, $\bs G$ is necessarily CP-parabolic almost surely.
%\end{corollary}
%\begin{proof}
%	If $\bs G$ is CP-hyperbolic on event $A$, then $[\bs G, \bs o]$ conditioned on $A$ is a unimodular triangulation which is CP-hyperbolic almost surely. So it has a stationary circle packing in the hyperbolic plane by Theorem~??? (later) of~\cite{BeTi19}. This contradicts Theorem~\ref{thm:scp2} unless $\myprob{A}=0$.
%\end{proof}

The following are examples satisfying the assumption of Theorem~\ref{thm:scp2}. The first example is easier to analyze, but has unbounded degree. It is modified in the second example to have bounded degrees.
\begin{figure}[t]
	\begin{center}
		\includegraphics[width=.9\textwidth]{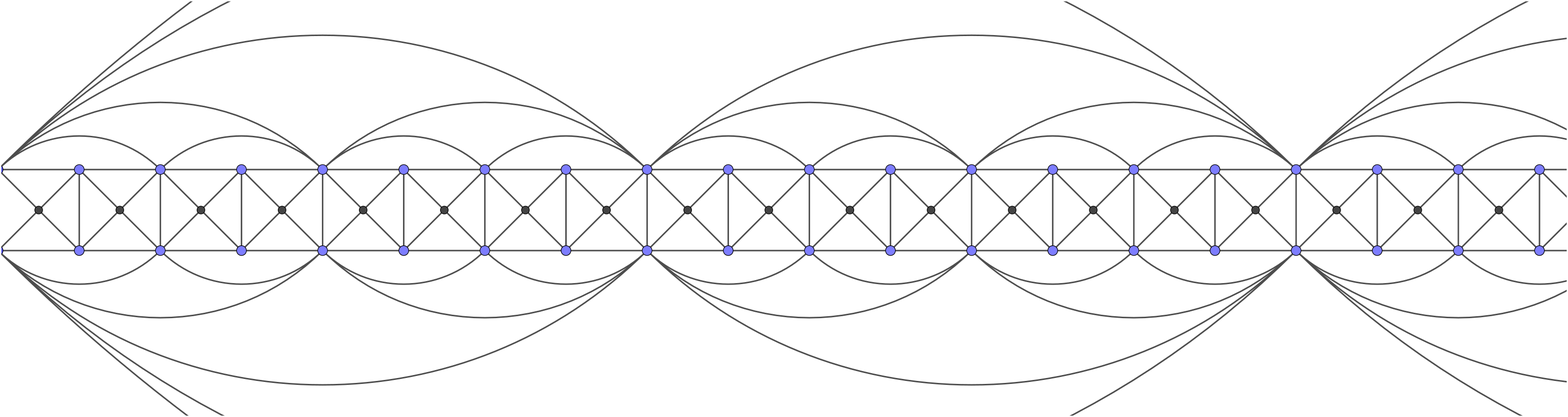}
		\caption{The graph of Example~\ref{ex:sCP}.}
		\label{fig:scp}
	\end{center}
\end{figure}

\begin{example}
	\label{ex:sCP}
	Let $u_0,u_1,\ldots$ be i.i.d. uniform random numbers in $\{0,1\}$ and let $a_n:=\sum_{i=0}^{n-1} 2^iu_i$. For every $n> 0$ and $m\in\mathbb Z$, connect the point $(m 2^n + a_n, 0)$ to $((m+1)2^n+a_n,0)$ by a semicircular arc in the upper half plane. For $n=0$, do the same but use straight line segments instead of arcs (see Figure~\ref{fig:scp}). Note that these arcs do not intersect and every arc in level $n$ contains exactly two arcs in level $n-1$ (if $n\geq 1$) and is contained in one arc in level $n+1$. So a triangulation of the upper half plane is obtained, namely $\bs G_0$. Reflect this triangulation about the line $y=-\frac 1 2$ and connect every point $(m,0)$ to $(m,-1)$. Finally, in each of the resulting squares in the band $-1\leq y\leq 0$, add a new vertex in the center and connect it to all four corners to obtain a triangulation of the whole plane, namely $\bs G$. Note that this triangulation is invariant under the translation by vector $(m,0)$ for every $m\in\mathbb Z$. Using this, one can obtain that if $\bs o$ is a random point in $\{(0,0), (0,-1), (\frac 12, -\frac 12) \}$ chosen uniformly, then $[\bs G, \bs o]$ is unimodular. It is easy to see that $\bs G$ is locally finite and has exactly one automorphism other than the identity. So, by Theorem~\ref{thm:scp2}, it cannot have a stationary CP. Also, it is CP-parabolic a.s. (otherwise, by~\cite{BeTi19}, it would have a stationary CP). In fact, it will be proved in Section~\ref{sec:anotherapproach} that this graph does not have any point-stationary CP as well.
	%
	%(later: use the corollary to deduce that it is amenable) Note that $\bs G$ is amenable since its vertices can be totally ordered in a measurable way (later: ref) (let $(m,0)<(m,-1)<(m+\frac 12, -\frac 12)<(m+1,0)$). Therefore, by Theorem~??? (later) of~\cite{AnHuNaRa16}, $\bs G$ is CP-parabolic.
\end{example}

\begin{figure}[t]
	\centering
	\includegraphics[width=.6\textwidth]{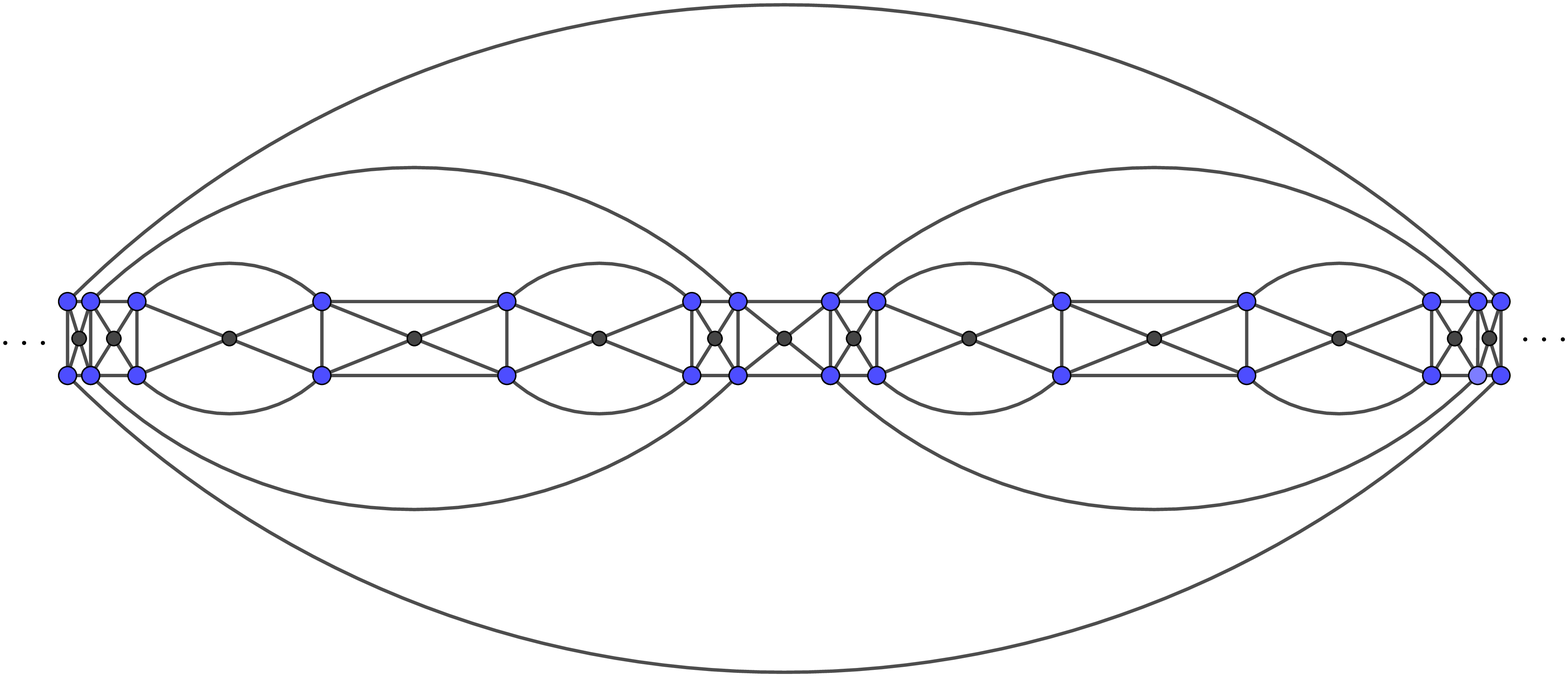}
	\caption{Part of the graph of Example~\ref{ex:sCPbounded} (the vertices inside the hexagons are not shown).}
	\label{fig:scp2}
\end{figure}

\begin{example}
	\label{ex:sCPbounded}
	Modify the above example by connecting $(m 2^n + a_n - 2^{-n-2}, 0)$ to $((m+1)2^n+a_n+2^{-n-2},0)$ for all $n\geq 0$ and $m\in\mathbb Z$. Then, connect every vertex to the next according to the increasing order on the line, except where there is already a semicircular edge (see Figure~\ref{fig:scp2}). In every resulting hexagon, add a vertex inside the hexagon and connect it to all six vertices. Then, reflect this graph about the line $y=-\frac 12$ and continue in the same manner. Note that one can correspond to every vertex in this graph a unique $k\in \mathbb Z$ naturally (e.g., the closes point of the form $(k,0)$). Using this, one can bias the probability measure and change a new root such that a unimodular graph is obtained (see the examples of~\cite{processes} or the general \textit{unimodular extension} in~\cite{shift-coupling}). Similarly to Example~\ref{ex:sCP}, this example has no stationary CP.
	%
	%One can make this graph unimodular as follows: Note that for every $k\in\mathbb Z$, if $(k,0)$ has $d(k)$ neighbors in $\bs G_0$ in Example~\ref{ex:sCP}, then in the present example, one can correspond it to $2d(k)-2$ vertices on the $x$ axis, one vertex in the center of the hexagon \textit{right above} it, the reflected images of these points, and $2d(k)-2$ vertices on $y=-\frac 12$. Using this, one can bias the probability measure by $6d(0)-4$ and change the root such that a unimodular graph is obtained (later: ref). Similarly to Example~\ref{ex:sCP}, this example has no stationary circle packing.
\end{example}

\begin{proof}[Proof of Theorem~\ref{thm:scp}]
	The claim is directly implied by Theorem~\ref{thm:scp2} and Example~\ref{ex:sCPbounded}.
\end{proof}

\section{An Example with No Point-Stationary CP}
\label{sec:pscpCounter}

This section constructs an example satisfying Theorem~\ref{thm:pscp}. The construction is by a modification of Example~\ref{ex:sCP}. In that Example, consider the subgraph formed by the vertices and edges in the upper closed half plane. %Note that in the planar dual of this subgraph, every triangle has 3 adjacent triangles except the triangles adjacent to the $x$ axis. Indeed, the dual graph is the canopy tree: Every triangle has one \textit{parent triangle} and (except in the lowest generation) two \textit{offspring triangles}. 
Not that the dual of this subgraph is the canopy tree: Each triangle has one \textit{parent triangle} and (except in the lowest level) two \textit{offspring triangles}. In this section, a similar construction will be provided by letting each triangle have either 2 or 0 offspring triangles randomly, as described below. %In fact, the dual graph will be an \textit{eternal Galton-Watson tree}~\cite{eft}.

\begin{figure}[t]
	\centering
	\includegraphics[width=.9\textwidth]{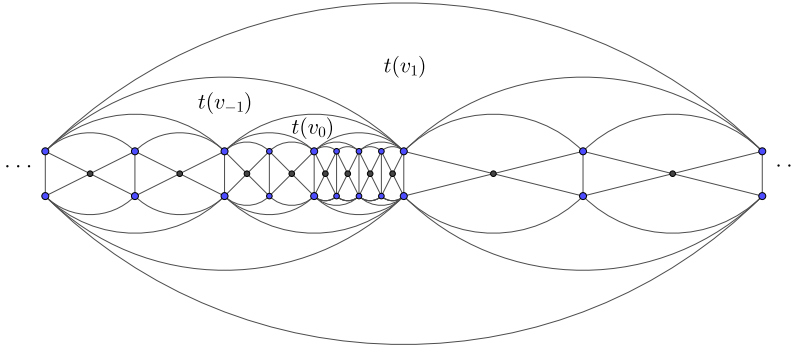}
	\caption{Part of the graph $[\bs G, \bs o]$. The upper half, excluding the vertices on the axis of symmetry, is $\bs G_0$. If $v_a:=(a,0)$, then the  vertices of $\bs G_0$ shown in this figure are $v_{-3},v_{-2},v_{-1},v_{-1/2}, v_0, v_{1/4}, v_{1/2}, v_{3/4}, v_1, v_3, v_5$ from left to right. One has $p(v_{-2})=p(v_0)=v_{-1}$, $p(v_{-1})=p(v_3)=1$, $p(v_{-1/2}) = p(v_{1/2})=0$ and $p(v_{1/4})=p(v_{3/4})=1/2$. Also, $v_1\in L_{-2}$, $\{v_{-1},v_3\}\subseteq L_{-1}$, $\{v_{-2},v_0\}\subseteq L_0$, $\{v_{-1/2},v_{1/2}\}\subseteq L_1$ and $\{v_{1/4},v_{3/4}\}\subseteq L_2$.
		%The dashed line is the $x$ axis and the subgraph above the $x$ axis is $\bs G_0$. In this figure, if the vertices of the form $(a, 0)$ are denoted by $a$, then the vertices of $\bs G_0$ are $\{-3,-2,-1,-1/2, 0, 1/4, 1/2, 3/4, 1, 3, 5 \}$ from left to right.  
	}
	\label{fig:pscp}
\end{figure}

First, \mar{An easier description: In an ordered EGW drawn in the plane, add a new vertex and connect it to all leaves by non-intersecting curves. The dual of this graph is $\bs G$!} a graph $\bs G_0$ is constructed in the upper half plane such that its faces are triangles and all edges are semicircles (see Figure~\ref{fig:pscp}). 
In this graph, each triangle is of the form $t_{a,b}$, where $t_{a,b}$ is the triangle with vertex set $\{(a,0),(b,0), (\frac{a+b}{2},0)\}$. Every such triangle has a \textit{parent triangle}, which is either $t_{2a-b,b}$ or $t_{a,2b-a}$. Also, $t_{a,b}$ will have either zero or two \textit{offspring triangles}, which are $t_{a,(a+b)/2}$ and $t_{(a+b)/2,b}$. To construct the graph, start with the \textit{root triangle} $t_{-1,1}$. 
Then, draw its parent triangle by choosing between the two options randomly, each case with probability $1/2$. Do the same for the new triangle and continue inductively to construct the \textit{ancestor line} of the root triangle. Each one of these triangles, except the root triangle, has already one offspring triangle. Draw the other offspring triangle as well. In the next step, for each one of the triangles (including the root triangle), either add the two offspring triangles or add none of them (each case with probability $1/2$ independently of the other triangles) and continue inductively for the new triangles. 
By continuing this process, the graph drawn in the plane will converge to a graph which we call $\bs G_0$. Indeed, the triangles \textit{below} every given triangle form a Galton-Watson process which is critical, and hence, will extinct a.s. So every triangle has finitely many descendants. Also, keep track of the order of the vertices by assigning marks; e.g., by directing the edges from left to right. In this construction, the dual graph (excluding the infinite face) is an (ordered) eternal Galton-Watson tree defined in~\cite{eft} (also studied previously in~\cite{fringe}). By \cite{eft}, the dual graph is unimodular and it implies that $[\bs G_0,\bs o_0]$ is also a unimodular graph, where $\bs o_0:=0$ (but the constructed planar embedding is not unimodular). Finally, as in Example~\ref{ex:sCP}, reflect $\bs G_0$ about the line $y=-\frac 1 2$, connect every vertex $(a,0)$ to $(a,-1)$, and in every resulting quadrilateral, add a new vertex in the center and connect it to all four corners. Call this planar triangulation $\bs G$. If $\bs o$ is chosen from $(0,0), (0,-1)$ and the first vertex on the half line $\{x>0, y=-\frac 12\}$ randomly and uniformly, then $[\bs G, \bs o]$ is a unimodular planar triangulation (but the embedding is not unimodular). %Let $\bs G_1$ be the reflected image of $\bs G_0$ and let $\bs S$ be the set of the other vertices which are in 

\begin{proof}[Proof of Theorem~\ref{thm:pscp}]
	We claim that $[\bs G,\bs o]$, constructed above, does not have any point-stationary CP. An example with bounded degrees can be constructed by modifying $[\bs G, \bs o]$ similarly to Example~\ref{ex:sCPbounded}. For simplicity of the proof, the modifications are only summarized in Remark~\ref{rem:pscpbounded}.
	
	%Note that it is amenable and hence, is CP-parabolic (later: ref). Assume that $[\bs G, \bs o]$ has a point-stationary circle packing in the Euclidean plane. 
	Assume $[\bs G, \bs o]$ has a point-stationary circle pacing.
	Since $\bs G$ has a unique automorphism other than the identity and it has infinitely many fixed points, by Lemma~\ref{lem:automorphism}, the set of circles should be invariant under a unique reflection, namely through line $\bs l$. This reflection swaps the circles corresponding to vertices $(a,0)$ and $(a,-1)$. Since these circles are tangent, both of them are tangent to $\bs l$. Therefore, a CP of $\bs G_0$ is obtained such that all circles are tangent to $\bs l$. % and are in one side of $\bs l$. 
	It can be seen that by conditioning on the event $\bs o=0$, a point-stationary CP of $[\bs G_0,\bs o_0]$ is obtained (indeed, $\bs G_0$ is an \textit{equivariant subgraph} of $\bs G$, see~\cite{I}). This contradicts Proposition~\ref{prop:pscp} below. So the claim is proved.
\end{proof}

\begin{proposition}
	\label{prop:pscp}
	The unimodular planar graph $[\bs G_0, \bs o_0]$ has no point-stationary circle packing in which all circles are tangent to a common line.
\end{proposition}
\begin{proof}
	The following notations are used in the proof.
	Every vertex $v$ of $\bs G_0$ is the middle vertex of some triangle, denoted by $t(v)$. So the genealogical structure of the triangles induces a similar structure on the vertices. For each vertex $v$, let $p(v)$ denote its parent (see Figure~\ref{fig:pscp}). Let $l(v,w)$ denote the number of generations between $v$ and $w$, which is defined by the equations $l(v,v):=0$ and $l(v,p(v))=-1$. 
	Observe that $p(v)$ is one of the vertices of $t(v)$ and the other vertex of this triangle satisfies $l(v,w)\leq -2$. 
	Let $L_n:=\{v: l(\bs o_0,v)=n\}$ be the $n$'th generation of the vertices (w.r.t. the root).  % and call it the $n$-th \textit{foil}. 
	The ascending order on the $x$ axis (in the definition of $\bs G_0$) induces an order on $L_n$. By~\cite{eft}, $L_n$ is nonempty for every $n\in\mathbb Z$ and is infinite from both sides. For $v\in L_n$, let $\tau(v)$ denote the next vertex in $L(v)$ according to this order. For $m\in \mathbb Z$, let $\tau^m(v)$ be defined by the $m$-fold composition of $\tau$.
	
	Assume there exists a point-stationary circle packing $\bs C$ such that all circles are tangent to a common line. For vertices $v$, let $\bs r(v)$ denote the radius of $C_v$. % and $\bs q(v):=\frac{\bs r(v)}{1+\bs r(v)}$. 
	%One can assume the common line is the horizontal line $y=-\bs r(\bs o_0)$. 
	Note that $\tau$ is a bijective map on the vertices which is defined independently of the root. Therefore, the distribution of $[\bs G_0,\bs o_0]$ is invariant under moving according to $\tau$; i.e., $[\bs G_0,\tau(\bs o_0)]$ has the same distribution as $[\bs G_0,\bs o_0]$ (see Proposition~3.6 of~\cite{eft}). By point-stationarity, the same holds for the circle packing of $\bs G_0$ (see Mecke's point stationarity in~\cite{ThLa09}). So, by letting $\varphi(x):=\frac x{1+x}$,  Birkhoff's pointwise ergodic theorem implies that the spatial average of $\varphi(\bs r(w))$ for $w\in L_0$, defined by the following formula, exists almost surely.
	\begin{equation}
		\label{eq:ave}
		\ave(\varphi\circ\bs r; L_0):=\lim_{M\to\infty} \frac 1 {2M+1} \sum_{m=-M}^M \bs \varphi(\bs r(\tau^m(\bs o_0))).
	\end{equation}
	By unimodularity, the same holds for all level sets (see e.g., Lemma~2.6 of~\cite{eft}). Hence, $\ave(\varphi\circ\bs r;L_n)$, which is defined for each $n$ by a similar formula, exists a.s. 
	Note that $\bs r(v)< \bs r(p(v))$ for every $v$. Therefore, heuristically, one can expect that $\ave(\varphi\circ\bs r; L_1)< \ave(\varphi\circ \bs r; L_0)$. The potential obstacle is that those vertices in $L_0$ with no offspring have so small radius that affects the average $\ave(\varphi\circ \bs r; L_0)$. In the next steps, it will be shown that the heuristic is indeed true by showing that the number of offsprings of the vertices in $L_0$ are mutually independent and also independent from their radii.
	%Later: give intuition that $\ave(\varphi\circ\bs r;L_1)<\ave(\varphi\circ\bs r;L_0)$ and say what is the problem to prove it and what is the idea to resolve it.
	
	%\textbf{First Proof.}\mar{later: choose one of the proofs.}
	Let $\bs H_0$ be the subgraph of $\bs G_0$ spanned by $\{v: l(\bs o_0,v)\leq 0\}$. Note that $\bs G_0$ is obtained by appending to $\bs H_0$ finite trees at the nodes of $L_0$. For $v\in L_0$, let $D(v)$ be the branch that is appended at $v$ (which consists of the descendants of $v$).
	Also, let $\bs\epsilon(v):=1$ if $v$ has two offsprings and $\bs\epsilon(v):=0$ otherwise.
	Let $\bs H_1$ be the graph obtained by extending $\bs H_0$ by appending two offsprings to every vertex $v\in L_0$ in the same manner as the definition of $\bs G_0$ ($\bs H_1$ is not a subgraph of $\bs G_0$). The restriction of the circle packing $\bs C$ to $\bs H_0$ uniquely extends to a CP of $\bs H_1$ in which all circles are tangent to the common line. For $v\in L_0$, let $\bs s(v):=\frac 12 (\varphi(a)+\varphi(b))$, where $a$ and $b$ are the radii of the circles corresponding to the offsprings of $v$. One can show that $\ave(\varphi\circ\bs r;L_1)=\ave(\bs s; \{\bs\epsilon=1 \})$, where $\{\bs\epsilon=1\}$ stands for the set $\{v\in L_0: \bs\epsilon(v)=1 \}$. The ergodic theorem also implies that $\ave(\bs s; \{\bs\epsilon=1 \}) = \ave(\bs\epsilon\bs s; L_0)/\ave(\bs\epsilon; L_0) = 2\ave(\bs\epsilon\bs s; L_0)$ a.s.
	On the other hand, the definition of $\bs s$ implies that $\bs s<\varphi\circ\bs r$. Therefore,
	%$\ave(\varphi\circ\bs r;L_1)=\ave(\bs\epsilon\bs s; L_0)$.  Therefore,
	\begin{equation}
	\label{eq:<0}
	\ave(\varphi\circ\bs r;L_1)=2\ave(\bs\epsilon\bs s; L_0) < 2\ave (\bs\epsilon \varphi\circ\bs r;L_0), \quad a.s.
	\end{equation}
	The above inequality holds because Birkhoff's pointwise ergodic theorem implies that for every stationary sequence $(f(v))_{v\in L_0}$, $\ave(f;L_0)$ is equal to the conditional expectation of $f(\bs o_0)$ given the invariant sigma-field corresponding to the action of $\tau$. % (later: ref). 
	Hence, if $f<g$, then $\ave(f;L_0)<\ave(g;L_0)$ a.s.

	Conditional to $[\bs H_0,\bs o_0]$, the sequence $(D(v))_{v\in L_0}$ is an i.i.d. sequence of random trees. % $(D(\tau^{(m)}(\bs o_0)))_{m\in \mathbb Z}$
	Let $Z:=\ave(\varphi\circ\bs r;L_0)$. Note that $Z$ depends on $\bs C$ and the root, but would be the same if any of the vertices in $L_0$ would be the root; i.e., $Z$ is invariant under the action of $\tau$. Therefore, Lemma~\ref{lem:ergodic} implies that conditional to $[\bs H_0,\bs o_0]$, $Z$ is independent from the sequence $(D(v))_{v\in L_0}$. Since the sequence is also independent from $[\bs H_0,\bs o_0]$, it follows that $(D(v))_{v\in L_0}$ is independent from the pair $\left([\bs H_0,\bs o_0], Z \right)$. As a result, %conditional to the invariant sigma-field, 
	the sequence $\bs\epsilon(\cdot)$ is independent from the sequence $\varphi\circ\bs r(\cdot)$ (since the latter is determined by $[\bs H_0,\bs o_0]$ and $Z$). Therefore, Lemma~\ref{lem:ave} implies that 
	\[
		\ave(\bs\epsilon \varphi\circ\bs r) = \ave(\bs\epsilon)\cdot \ave(\varphi\circ\bs r) = \frac 12 \ave(\varphi\circ\bs r),\quad a.s.
	\]
	So~\eqref{eq:<0} implies that $\ave(\varphi\circ\bs r;L_1)< \ave(\varphi\circ\bs r;L_0)$ a.s., as claimed.
	By unimodularity, one gets 
	\[
		\forall n\in\mathbb Z: \ave(\varphi\circ\bs r;L_{n+1})< \ave(\varphi\circ\bs r;L_n), \quad a.s.
	\]
	Now, for any given $a\in\mathbb R$, one can distinguish the foil $L_N$ defined by $N:=\min\{n: \ave(\varphi\circ\bs r;L_n)<a \}$ (there is at least one $a$ such that the set under minimum is nonempty with positive probability).	
	This contradicts Lemma~\ref{lem:foilselection} below. So the claim is proved. %So it remains to prove~\eqref{eq:<}.
\end{proof}

The following lemmas are used in the above proof.
%In \mar{1. later: keep this?\\ 2. Say that the foils are indistinguishable} Lemma~\ref{lem:foilselection}, by selecting a foil of $[\bs G_0,\bs o_0]$, we mean a unimodular marked graph such that the unmarked graph has the same distribution as $[\bs G_0, \bs o_0]$ and the set of vertices with mark 1 is a foil of $\bs G_0$. By Remark~\ref{rem:equivprocess}, this is equivalent to the following: In every sample of $[\bs G_0, \bs o_0]$, select a random foil without looking at the root such that the distribution of the selection is invariant under isometries (and some measurability criterion holds).

\begin{lemma}
	\label{lem:foilselection}
	For the graph $[\bs G_0, \bs o_0]$ defined in this section, there is no way to select one of the foils $\{L_n: n\in \mathbb Z\}$ (possibly using extra randomness), in an event with positive probability, such that unimodularity is preserved (see Remark~\ref{rem:equivprocess}).
\end{lemma}
\begin{proof}
	Let $g(u,v):=1$ if $v$ is the first ancestor of $u$ in the selected foil and $g(u,v):=0$ otherwise. Then ${\sum_v g(\bs o_0,v)}\leq 1$ a.s. Also, $\omid{\sum_v g(v,\bs o_0)} = \omid{f(\bs o_0)\card{D(\bs o_0)}}$, where $D(\bs o_0)$ is the set of descendants of $v$ and $f(u):=1$ if $u$ is in the selected foil. 
	%Let $f(v):=1$ if $v$ is in the selected foil and $f(v):=0$ otherwise. For every $v\in \bs G_0$, send unit mass from $v$ to the first ancestor of $v$ in the selected foil (if exists any). Then the outgoing mass from the root is at most 1. The incoming mass to the root is $f(\bs o_0)\card{D(\bs o_0)}$, where $D(\bs o_0)$ is the set of descendants of $v$. 
	By Lemma~\ref{lem:ergodic}, $f(\bs o_0)$ is independent from $\card{D(\bs o_0)}$. Therefore, $\omid{f(\bs o_0)\card{D(\bs o_0)}}= \omid{f(\bs o_0)}\omid{\card{D(\bs o_0)}}=\infty$, where the latter is because $D(\bs o_0)$ is a critical Galton-Watson tree, which gives $\omid{\card{D(\bs o_0)}}=\infty$. This contradicts the mass transport principle.
\end{proof}

\begin{lemma}
	\label{lem:ergodic}
	Assume $Z$ and $(X_n)_{n\in\mathbb Z}$ are random variables on the same probability space. Let $Y_n:=X_{n+1}$ for each $n$. If 
	%If $Z$ and $(X_n)_{n\in\mathbb Z}$ are defined on the same probability space, 
	$(X_n)_n$ is ergodic and $(Z,(Y_n)_n)$ has the same distribution as $(Z,(X_n)_n)$, then $Z$ is independent of $(X_n)_n$.
\end{lemma}
\begin{proof}
	Consider an event of the form $Z\in A$. % and let $\probCond{Z\in A}{(X_n)_n} =: f((X_n)_n)$.
	Then $\probCond{Z\in A}{(X_n)_n}$ is a function of $(X_n)_n$ which is invariant under shift. So ergodicity implies that it is constant a.s. and the constant is $\omid{\probCond{Z\in A}{(X_n)_n}} = \myprob{Z\in A}$. This implies the claim.
\end{proof}

\begin{lemma}
	\label{lem:ave}
	%\mar{Do we need to assume $X_1Y_1\in L^1$?}
	Let $(X_n)_{n\in\mathbb Z}$ and $(Y_n)_{n\in\mathbb Z}$ be independent stationary sequences of random variables such that $X_1,Y_1\in L^1$. Then $\ave(XY)=\ave(X)\ave(Y)$ a.s.
\end{lemma}
\begin{proof}
	\mar{later: is the claim known? Is there an easier proof? Are the explanations sufficient?}
	If the sequence $(X_n,Y_n)_n$ is ergodic, then the ergodic theorem implies that $\ave(X)=\omid{X_1}$, $\ave(Y)=\omid{Y_1}$  and $\ave(XY)=\omid{X_1Y_1}$ a.s. and the claim follows. 
	In general, let $\int \mu^{(1)}_z d\nu_1(z)$ and $\int \mu^{(2)}_z d\nu_2(z)$ be the ergodic decompositions of the distributions of $(X_n)_n$ and $(Y_n)_n$ respectively. Then $\int\int \mu^{(1)}_z\otimes\mu^{(2)}_t d\nu_1(z)d\nu_2(t)$ is a decomposition of the distribution of $(X_n,Y_n)_n$ and each component $\mu^{(1)}_z\otimes\mu^{(2)}_t$ is ergodic. So it is an ergodic decomposition. By the first part of the proof, the claim holds for each ergodic component. This implies the claim $\ave(XY)=\ave(X)\ave(Y)$ a.s. since it is a property of the samples of the processes.
	%In general, by considering the ergodic components of $(X_n)_n$ and $(Y_n)_n$ and their product measures, the ergodic components of $(X_n,Y_n)_n$ are obtained. So the claim holds for the ergodic components of $(X_n,Y_n)_n$ and this implies the claim.
\end{proof}

\begin{remark}
	The %\mar{later: I think that since Example~\ref{ex:sCP} also does not have psCP, $\ave(\cdot; L_0)$ also does not exist.}
	existence of the spatial averages in~\eqref{eq:ave} is heavily based on the assumption of point-stationarity, which was shown to be impossible. Indeed,
	%Lemma~\ref{lem:equivCP} and the proof of Proposition~\ref{prop:scpStronger} implies that  $\ave(\bs r; \bs L_0)$ does not exist, and also, $\ave(\varphi\circ(c\bs r); \bs L_0)$ cannot exist for all $c>0$ (otherwise, one would obtain an equivariant CP of $\bs L_0$ by a suitable scaling).
	we guess that in every CP of $[\bs G_0,\bs o_0]$, $\ave(\varphi\circ\bs r; L_0)$ does not exist. A similar statement that can be proved now, is that Proposition~\ref{prop:pscp} and Lemma~\ref{lem:equivCP} imply that $\ave(\bs r; \bs G_0)$ does not exist (otherwise, one would obtain an equivariant CP by a suitable scaling). Also, $\ave(\varphi\circ(c\bs r); \bs G_0)$ cannot exist for all $c>0$.
\end{remark}

\begin{figure}[t]
	\centering
	\includegraphics[width=.9\textwidth]{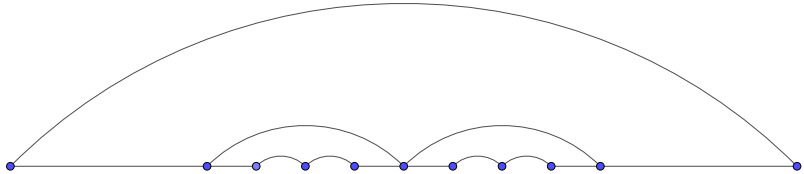}
	\caption{A pentagon in Remark~\ref{rem:pscpbounded} and its two offspring pentagons. In each pentagon, two edges can be drawn by line segments instead of semicircles.}
	\label{fig:pscp2}
\end{figure}

\begin{remark}
	\label{rem:pscpbounded}
	To construct an example with bounded degrees, one can modify $[\bs G, \bs o]$ similarly to Example~\ref{ex:sCPbounded}. For this, replace every triangle $t_{a,b}$ in the construction of $\bs G_0$ with a pentagon with vertex set $\{(\frac{ka+(4-k)b}{4},0): k=0,\ldots, 4 \}$. The potential offsprings of $t_{a,b}$ are $t_{(3a+b)/4, (a+b)/2}$ and $t_{(a+b)/2,(a+3b)/4}$ (see Figure~\ref{fig:pscp2}). 
	%three semicircular edges and two line segments such that if the upper arc of the pentagon is connecting $(a,0)$ and $(b,0)$, then another semicircular arcs connects $(\frac {3a+b}4, 0), (\frac{a+b}{2},0)$ and the other connects $(\frac{a+b}{2},0),(\frac{a+3b}{4},0)$. 
	By adding a new vertex inside each pentagon, one can obtain a triangulation of a half plane. The reader can verify that this graph can be obtained from the previous graph by splitting every vertex into finitely many vertices (with finite expectation) and adding new edges. Therefore, one can make it unimodular by biasing and changing the root similarly to Example~\ref{ex:sCPbounded}. Similarly to Proposition~\ref{prop:pscp}, one can prove that this triangulation does not have any point-stationary CP such that all circles (except those corresponding to the centers of the pentagons) are tangent to a common line (note that every pentagon $t_{a,b}$ has a \textit{middle vertex} $(\frac{a+b}{2},0)$, and so, the middle vertices of the pentagons are structured as an Eternal Galton-Watson tree). Finally, by reflecting this graph through the line $y=-\frac 12$ and adding new vertices on this line, one can obtain a whole-plane triangulation (similarly to Example~\ref{ex:sCPbounded}) and deduce that it does not have any point-stationary CP (similarly to the proof of Theorem~\ref{thm:scp}).
\end{remark}

\section{An Approach By Finite Approximations}
\label{sec:anotherapproach}
In this section, to study whether a unimodular planar graph has a point-stationary CP or not, we will use approximations of the graph by finite planar graphs. A particular result is that the graph constructed in Section~\ref{sec:scpCounter} has no point-stationary CP as well (this cannot be proved by the method of Section~\ref{sec:pscpCounter}). In general, we are focused on the CP-parabolic case, since the other cases always have stationary CPs in the hyperbolic plane. It is worth to mention that CP-parabolic unimodular planar graphs can always be approximated by finite planar graphs \cite{AnHuNaRa18}; i.e., are \textit{sofic}.

%Many considerations should be treated here. %There are many details to make this idea precise.

%In what follows, the idea is to replace spatial averages by averages over finite graphs that approximate $[\bs G, \bs o]$.

Here, it is convenient to consider circle packings up to similarities of the plane. So we define: %For simplicity, we restrict attention to triangulations only (see Remark~??? for the general case). 

\begin{definition}
	Let $G$ be a planar graph and let $P=\{C_v\}_{v\in V(G)}$ be a CP of $G$. Let $r(v)$ be the radius of $C_v$ and $d(u,v)$ be the distance of the centers of $C_u$ and $C_v$. The \defstyle{CP-cocycle} of $P$ is the map $c:G\times G\to\mathbb R^2$ defined by $c(u,v):=\big(r(v)/r(u), d(u,v)/r(u)\big)$. 
	If $G$ is a triangulation, one can drop the second term and define $c(u,v):=r(v)/r(u)$.
	\\
	In addition, if $[\bs G, \bs o]$ is a unimodular planar graph and $\bs c$ is a random CP-cocycle of $\bs G$, then it is called an \defstyle{equivariant CP-cocycle} if by regarding $\bs c(u,v)$ as the mark of the pair $(u,v)$, a unimodular marked graph (i.e., a unimodular network) is obtained (see Remark~\ref{rem:equivprocess}). Convergence of a sequence $[\bs G_n, \bs o_n; \bs c_n]$, where $\bs c_n$ is a (random) CP-cocycle of $[\bs G_n, \bs o_n]$, can be defined using the notion convergence of marked graphs (see e.g., \cite{processes} and \cite{I}).
\end{definition}

Note the the CP-cocycle of $P$ uniquely determines $P$ up to similarities. The term \textit{cocycle} is chosen since the marking $c(u,v):=r(v)/r(u)$ satisfies $c(u,v)c(v,w)=c(u,w)$. 

%\begin{remark}
	One can show that every point-stationary CP of $[\bs G, \bs o]$ determines an equivariant CP-cocycle of $[\bs G, \bs o]$, but the converse is not necessarily true (it is indeed true for finite graphs). See e.g., Examples~\ref{ex:Z-1} and~\ref{ex:Z-iid}. % For instance, the cocycle $c(m,n):=2^{n-m}$ is a CP-cocycle of the graph $[\mathbb Z,0]$, but is not realized by any point-stationary CP.
%\end{remark}

\begin{example}
	\label{ex:uniquecocycle}
	Assume $[\bs G, \bs o]$ is a unimodular one-ended triangulation. By uniqueness of a CP for $\bs G$ up to similarities, one gets that $\bs G$ has a unique CP-cocycle a.s. Since it is a deterministic (and measurable) function of $\bs G$ and does not depend on the root, it is an equivariant CP-cocycle.
\end{example}

\begin{definition}
	Let $P$ be a circle packing of a finite planar graph and $0<\epsilon< 1$. % let $q(P,\epsilon)$ be the $\epsilon$-quantile of the radii of the circle in $P$.
	If the radii of the circles are $r_1\leq\cdots\leq r_n$, let $q(P,\epsilon):=r_{\lceil \epsilon n\rceil}$ be the \textbf{$\epsilon$-quantile} of the radii, where $\lceil x\rceil=\min\{n\in\mathbb Z: n\geq x \}$. If $c$ is a CP-cocycle, $q(c,\epsilon)$ is not well defined, but the ratio $q(c,\epsilon)/q(c,\delta)$ is well defined.
\end{definition}

Note that the distribution of every unimodular finite graph is invariant under re-rooting to a new root chosen uniformly. The same holds for finite point-stationary point processes. Therefore, by Lemma~\ref{lem:equivCP}, every unimodular finite planar graph has a point-stationary CP, and hence, has an equivariant CP-cocycle.

\begin{theorem}
	\label{thm:tight}
	For $n\in\mathbb N$, let $[\bs G_n, \bs o_n; \bs c_n]$ be a unimodular finite planar graph equipped with an equivariant CP-cocycle. 
	Assume that this sequence converges weakly to a random marked graph $[\bs G, \bs o; \bs c]$. 
	If
	\begin{equation}
		\label{eq:quantiles}
		\forall \epsilon>0: \lim_{M\to\infty} \sup_n \myprob{\frac{q(\bs c_n,1-\epsilon)}{q(\bs c_n,\epsilon)}>M}= 0,
	\end{equation}
	then $[\bs G, \bs o]$ has a point-stationary circle packing such that its CP-cocycle is $\bs c$.
\end{theorem}

\begin{proof}
	Let $\bs P_n$ be a point-stationary CP of $[\bs G_n, \bs o_n]$ such that its CP-cocycle is $\bs c_n$. By applying a random isometry, one can assume the distribution of $\bs P_n$ is invariant under the isometries that fix the origin. Let $\bs r_n(\cdot)$ be the radii of the circles in $\bs P_n$.
	It can be seen the $\frac 1{\bs r_n(\bs o_n)}\bs P_n$ converges to a random circle packing of $[\bs G, \bs o]$ such that its CP-cocycle is $\bs c$. Here, $\bs P$ is regarded as a random function that assigns disjoint circles to the vertices of $[\bs G, \bs o]$. At the end of the proof, it will be shown that the set of circles in $\bs P$ does not have a limit point in the plane, and hence, $\bs P$ is an actual circle packing in the plane.
	 	
	Let $\bs q_n(\epsilon):=q(\bs P_n,\epsilon)$ and $\bs m_n:=q(\bs P_n,\frac 1 2)$, which are random variables. 
	Since the distribution of $[\bs G_n, \bs o_n]$ is invariant under random re-rooting, one has $\myprob{\bs r_n(\bs o_n)<\bs q_n(\epsilon)}\leq \epsilon$ and $\myprob{\bs r_n(\bs o_n)>\bs q_n(\epsilon)}\leq 1-\epsilon$.
	We claim that the sequence $\bs m_n/\bs r_n(\bs o_n)$ is tight in the open interval $(0,\infty)$. For every $\epsilon$ and $M$, one has
	\[
	\myprob{\frac{\bs m_n}{\bs r_n(\bs o_n)}>M}\leq \myprob{\bs r_n(\bs o_n)<\bs q_n(\epsilon)} + \myprob{\frac{\bs m_n}{\bs q_n(\epsilon)}>M}.
	\]
	Let $\delta>0$ be arbitrary.
	For $\epsilon:=\delta/2$, the first term in RHS is at most $\delta/2$. Also, by~\eqref{eq:quantiles}, one can choose $M$ large enough such that the second term in the RHS is less than $\delta/2$ for all $n$. So $\myprob{\bs m_n/\bs r_n(\bs o_n)>M}<\delta$ for all $n$. Similarly,
	\[
	\myprob{\frac{\bs m_n}{\bs r_n(\bs o_n)}<\frac 1 M}\leq \myprob{\bs r_n(\bs o_n)>\bs q_n(1-\epsilon)} +
	\myprob{\frac{\bs m_n}{\bs q_n(1-\epsilon)}>\frac 1{M}}
	\]
	and \mar{later: symbol for weak convergence} one can make the RHS arbitrarily small uniformly in $n$. So it is proved that $\bs m_n/\bs r_n(\bs o_n)$ is a tight sequence in $(0,\infty)$. Since the sequence $\frac 1{\bs r_n(\bs o_n)}\bs P_n$ is convergent, it is tight. So the sequence $(\frac 1{\bs r_n(\bs o_n)}\bs m_n,\frac 1{\bs r_n(\bs o_n)}\bs P_n)$ is also tight. Therefore, by passing to a subsequence if necessary, one can assume that the latter is convergent weakly. This implies that there is a random variable $\bs Z$ such that $(\frac 1{\bs r_n(\bs o_n)}\bs m_n,\frac 1{\bs r_n(\bs o_n)}\bs P_n)\Rightarrow (\bs Z, \bs P)$. In addition, by tightness in $(0,\infty)$, one has $0<\bs Z<\infty$ a.s. This gives that $\bs P'_n:=\frac 1{\bs m_n}\bs P_n\Rightarrow \frac 1{\bs Z}\bs P$. Since $\bs m_n$ does not depend on the root, $\bs P'_n$ is a point-stationary CP. This implies that $\bs P':=\frac 1{\bs Z}\bs P$ is also a point-stationary CP.
	
	It remains to prove that the set of circles in $\bs P'$ does not have a limit point a.s. Assume that this is not the case. So there exists $\lambda<\infty$ such that with positive probability, there are infinitely many circles in the ball with radius $\lambda \bs r(\bs o)$ centered at the origin, where $\bs r(\cdot)$ denotes the radii of the balls in $\bs P'$. Send unit mass from vertex $v$ to vertex $w$ if $\bs r(w)<\bs r(v)$ and the distance of the corresponding centers is at most $\lambda\bs r(v)$. Then the out-going mass is infinity with positive probability, while the incoming mass has a deterministic upper bound $(\lambda+1)^2$. This contradicts the mass transport principle for $[\bs G, \bs o; \bs r]$. So the claim is proved.
\end{proof}

\begin{conjecture}
	The converse of Theorem~\ref{thm:tight} also holds in the sense that if $[\bs G,\bs o]$ has a point-stationary CP such that its CP-cocycle is $\bs c$, then~\eqref{eq:quantiles} holds along some subsequence.
\end{conjecture}

We can prove a weaker statement as follows (Theorem~\ref{thm:quantiles}). First, the following definition is borrowed from~\cite{II}.
\begin{definition}
	\label{def:embedded}
	Let $[\bs G, \bs o]$ be a unimodular graph. An \defstyle{equivariantly embedded subgraph} of $[\bs G, \bs o]$ is a random subgraph $\bs H$ such that $\bs o\in V(\bs H)$ a.s. and 
	\begin{equation}
		\label{eq:embedded}
		\forall g: \omid{\sum_{v\in V(\bs H)} g(\bs G, \bs H, \bs o, v)} = \omid{\sum_{v\in V(\bs H)} g(\bs G, \bs H, v, \bs o)},
	\end{equation}
	where the conditions on $g$ are similarly to~\eqref{eq:mtp}. In addition, we always assume that $\bs H$ is the induced subgraph on its vertex set.
\end{definition}
One can deduce from~\eqref{eq:embedded} that $[\bs H, \bs o]$ is also a unimodular graph. See Examples~\ref{ex:Z-1}, \ref{ex:Z-iid} and the proof of Proposition~\ref{prop:scpStronger} for examples of this definition, together with applications of Theorem~\ref{thm:quantiles}.

\begin{theorem}
	\label{thm:quantiles}
	Let $[\bs G, \bs o]$ be a unimodular planar graph and $\bs G_n$ be an equivariantly embedded subgraph which is finite a.s. ($n=1,2,\ldots$). Assume $\bs c$ is an equivariant CP-cocycle of $[\bs G, \bs o]$ and let $\bs c_n$ be the restriction of $\bs c$ to $\bs G_n$. 
	If~\eqref{eq:quantiles} does not hold for this sequence, then $[\bs G, \bs o]$ has no point-stationary CP such that its CP-cocycle is $\bs c$.
	%If $[\bs G, \bs o]$ has a point-stationary CP such that its CP-cocycle is $\bs c$, then~\eqref{eq:quantiles} holds.
\end{theorem}
\begin{proof}
	Assume $\bs P$ is a point-stationary CP of $[\bs G, \bs o]$ such that its CP-cocycle is $c$. Let $\bs P_n$ be the restriction of $\bs P$ to $\bs G_n$. By~\eqref{eq:embedded}, it can be seen that $\bs P_n$ is a point-stationary CP of $[\bs G_n, \bs o]$ such that its CP-cocycle is $\bs c_n$. %Since $\bs G_n$ is finite, there exists a point-stationary circle packing $\bs P_n$ of $[\bs G_n, \bs o]$ such that its CP-cocycle is $\bs c_n$. 
	Let $\bs r_n(\cdot)$ (resp. $\bs r$) be the radii corresponding to $\bs P_n$ (resp. $\bs P$).
	%Given $[\bs G, \bs o]$ and $\bs G_n$, 
	Let $\bs o'_n$ be a random vertex of $\bs G_n$ chosen uniformly. By unimodularity, $[\bs G_n, \bs o'_n; \bs r_n]$ has the same distribution as $[\bs G_n, \bs o; \bs r_n]$. Now, let $\epsilon,\delta>0$ be arbitrary.
	Choose $M$ large enough such that $\myprob{\bs r(\bs o)\leq 1/M}<\delta$. 
	Since $\bs r_n(\bs o)=\bs r(\bs o)$, one gets
	\begin{eqnarray*}
		\delta &>&\myprob{\bs r(\bs o)\leq \frac 1M} = \myprob{\bs r_n(\bs o)\leq \frac 1 M} = \myprob{\bs r_n(\bs o')\leq \frac 1M}\\
		&=& \omid{\probCond{\bs r_n(\bs o')\leq \frac 1M}{[\bs G_n; \bs r_n]}}\\
		&= & \omid{\bs q_n^{-1}(\frac 1M)},
	\end{eqnarray*}
	where $\bs q_n^{-1}(x)=\sup\{\epsilon: \bs q_n(\epsilon)\leq x \}$.
	Therefore, Markov's inequality gives
	\[
		\myprob{\bs q_n(\epsilon)\leq\frac 1M} \leq \myprob{\bs q_n^{-1}(\frac 1M)\geq \epsilon} \leq \frac{\delta}{\epsilon}.
	\]
	Also, assume $M$ is so large that $\myprob{\bs r(\bs o)>M}<\delta$. Similarly, one gets
	\begin{eqnarray*}
		\delta &>&\myprob{\bs r(\bs o)>M} = \myprob{\bs r_n(\bs o)>M} = \myprob{\bs r_n(\bs o')>M}\\
		&=& \omid{\probCond{\bs r_n(\bs o')>M}{[\bs G_n; \bs r_n]}}\\
		&=& \omid{1-\bs q_n^{-1}(M)}.
	\end{eqnarray*}
	%where $\bs q_n^{-1}(M)=\sup\{\epsilon: \bs q_n(\epsilon)\leq M \}$.
	Therefore, Markov's inequality gives
	\[
		\myprob{\bs q_n(1-\epsilon)> M} \leq \myprob{\bs q_n^{-1}(M)\leq 1-\epsilon} = \myprob{1-\bs q_n^{-1}(M)>\epsilon} \leq \frac{\delta}{\epsilon}.
	\]
	By combining the two inequalities, one gets $\myprob{\bs q_n(1-\epsilon)/\bs q_n(\epsilon)>M^2}\leq 2\delta/\epsilon$. This implies~\eqref{eq:quantiles} and the claim is proved.
%	
%	So the definition of $q(\cdot)$ implies that $\myprob{\bs r_n(\bs o)<q(\bs P_n, \epsilon)}\leq \epsilon$ and $\myprob{\bs r_n(\bs o)>q(\bs P_n,\epsilon)}\leq 1-\epsilon$. 
%	Since $\bs r_n(\bs o)=\bs r(\bs o)$ (where $\bs r(\cdot)$ is the radii corresponding to $\bs P$), one obtains that $q(\bs P_n, \epsilon)$ is exactly the $\epsilon$-quantile of the random variable $\bs r(\bs o)$. This implies the claim.
\end{proof}

\begin{remark}
	\label{rem:quantiles-general}
	Note that $\bs G_n$ does not need to converge to $\bs G$ in Theorem~\ref{thm:quantiles}. Also, the proposition can be generalized to the following cases with the same proof:
	\begin{enumerate}[(i)]
		\item It is not needed that $\bs G_n$ is a connected subgraph. In general, the restriction of every CP of $[\bs G, \bs o]$ to $\bs G_n$, is a CP of $\bs G_n$. See e.g., Example~\ref{ex:Z-iid}.
		\item Even $[\bs G, \bs o]$ need not be connected. For this, one can assume that $[\bs G, \bs o]$ is a \textit{unimodular discrete metric space} \cite{I} equipped with an equivariant graph structure. This will be used in the proof of Proposition~\ref{prop:scpStronger}.
		\item In fact, $[\bs G, \bs o]$ need not be unimodular. It is only required that~\eqref{eq:embedded} holds for all of the subgraphs $\bs G_n$. Also, point-stationarity of a CP should be replaced by a similar formula to~\eqref{eq:embedded}. See e.g., the proof of Proposition~\ref{prop:scpStronger}, where conditioning breaks unimodularity.
	\end{enumerate}
\end{remark}

The following are basic examples for illustrations of the above result. The main application is in the proof of Proposition~\ref{prop:scpStronger}.

\begin{example}
	\label{ex:Z-1}
	Let $[\bs G, \bs o]:=[\mathbb Z, 0]$ and $\bs c(n,m):=2^{m-n}$. Then $\bs c$ is a CP-cocycle, but it is not the CP-cycle of any point-stationary CP of $[\bs Z, 0]$. To see this, let $\bs G_n$ be the subgraph induced by $\{0,1,\ldots,n\}-\bs U$, where $\bs U\in\{0,1,\ldots,n\}$ is uniformly at random. Observe that $q(\bs c_n, 3/4)/q(\bs c_n, 1/4)\approx 2^{n/2}$ and use Theorem~\ref{thm:quantiles}.
\end{example}

\begin{example}
	\label{ex:Z-iid}
	Let $[\bs G, \bs o]:=[\mathbb Z, 0]$ and $\bs c(n,n+1):=2^{\pm 1}$, each with probability $1/2$ independently for all $n\in\mathbb Z$. This can be extended to a CP-cocycle, but it is not the CP-cycle of any point-stationary CP of $[\bs Z, 0]$. To see this, let $\bs z_n:=\pm n$, each case with probability $1/2$, and $\bs G_n:=\{0, \bs z_n\}$. Since $\bs G_n$ has two elements, one gets $q(\bs c_n,3/4)/q(\bs c_n,1/4)= \max\{\bs c_n(0,\bs z_n), 1/ \bs c_n(0,\bs z_n) \}$. By the central limit theorem, $\frac 1{\sqrt n}\log\bs c(0,\bs z_n)$ converges to a non-degenerate normal distribution. This implies that~\eqref{eq:quantiles} does not hold and the claim is implied by Theorem~\ref{thm:quantiles}.
\end{example}

\begin{proposition}
	\label{prop:scpStronger}
	The unimodular graph $[\bs G, \bs o]$ of Example~\ref{ex:sCP} does not have any point-stationary circle packing.
\end{proposition}
\begin{proof}
	According to Example~\ref{ex:sCP}, there exists a circle packing $\bs P=\{C_v\}_{v\in V(\bs G)}$ of $[\bs G, \bs o]$ in the plane such that $C_{\bs o}$ is centered at the origin and has radius 1, and the CP is symmetric w.r.t. the line $y=-1$. Let $\bs c$ be the CP-cocycle of $\bs P$. By Example~\ref{ex:uniquecocycle}, it is an equivariant cocycle.
	
	Condition on the event $A:=\{a_0=1, \bs o=0\}$ and let $\bs G'$ be the subset consisting of the vertices $(2m,0)$. It can be seen that~\eqref{eq:embedded} holds for $[\bs G, \bs o]$ (conditioned on $A$) and the subset $\bs G'$. Here, $\bs G'$ does not have any edges, but one can regard it as a discrete metric space (see Remark~\ref{rem:quantiles-general}). By~\eqref{eq:embedded}, one can show that $[\bs G', 0]$ is a unimodular discrete space~\cite{I} (note however that $[\bs G, \bs o]$ conditioned on $A$ is not unimodular). Let $\bs P'$ (resp. $\bs c'$) be the restriction of $\bs P$ (resp. $\bs c$) to $\bs G'$. By~\eqref{eq:embedded}, it can be seen that $\bs c'$ is an equivariant CP-cocycle of $[\bs G', 0]$. If $[\bs G, \bs o]$ has a point-stationary CP, then, by~\eqref{eq:embedded}, $[\bs G',0]$ has a point-stationary CP such that its CP-cocycle is $\bs c'$. So it is enough to show that the latter is impossible. This will be shown by the generalization of Theorem~\ref{thm:quantiles} to unimodular discrete spaces (see Remark~\ref{rem:quantiles-general}). % (\mar{later: move this to a remark}it should be noted that Proposition~\ref{thm:quantiles} also holds for unimodular discrete spaces equipped with an equivariant graph structure).
	
	From now on, everything is conditioned on $A$ without mentioning repeatedly.
	The vertices of the graph $\bs G_0$ (see Example~\ref{ex:sCP}) have a genealogical structure similarly to the example in Section~\ref{sec:pscpCounter}. Let $t(v)$ and $p(v)$ be as in the proof of Proposition~\ref{prop:pscp}. Also, let $p'(v)$ be the third vertex of $t(v)$. For $n\geq 1$, let $\bs G^{(n)}$ be the subgraph of $\bs G_0$ consisting of the vertices that are below (or on) the triangle $t(p^{n}(0))$. A moment of thought shows that $\bs G^{(n)}$ is a deterministic graph \mar{later: figure + top circles + p(v), p'(v), offsprings, circles} (the structure of the triangles in $\bs G^{(n)}$ is a binary tree of depth $n$ in which the vertices at each level have an order, see Figure~\ref{fig:pscp3}). Also, by the construction of $\bs G_0$, it is straightforward to see that the vertex 0 is a uniform vertex of the last level of this graph. More precisely, if $\bs G'_n:=\bs G'\cap V(\bs G^{(n)})$ denotes the last level in $\bs G^{(n)}$, then $\bs G'_n$ is equivariantly embedded in $[\bs G_0, 0]$ (Definition~\ref{def:embedded}). Let $\bs P^{(n)}$ and $\bs c^{(n)}$ (resp. $\bs P'_n$ and $\bs c'_n$) be the restrictions of $\bs P$ and $\bs c$ to $\bs G^{(n)}$ (resp. $\bs G'_n$). See Figure~\ref{fig:pscp3}.
		\begin{figure}[t]
		\centering
		\includegraphics[width=.6\textwidth]{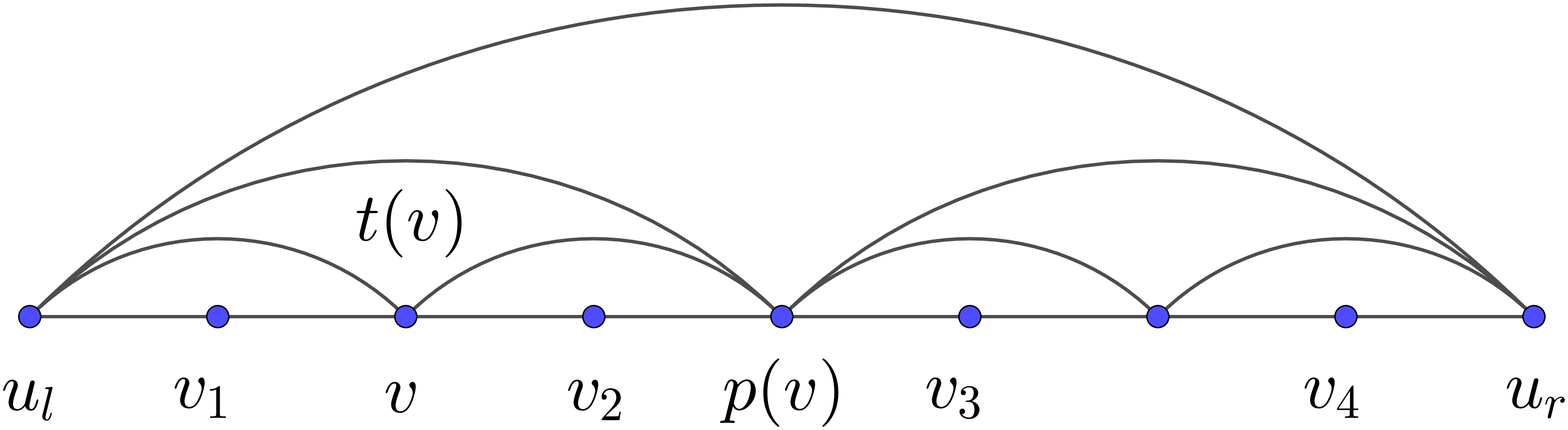}
		\includegraphics[width=.6\textwidth]{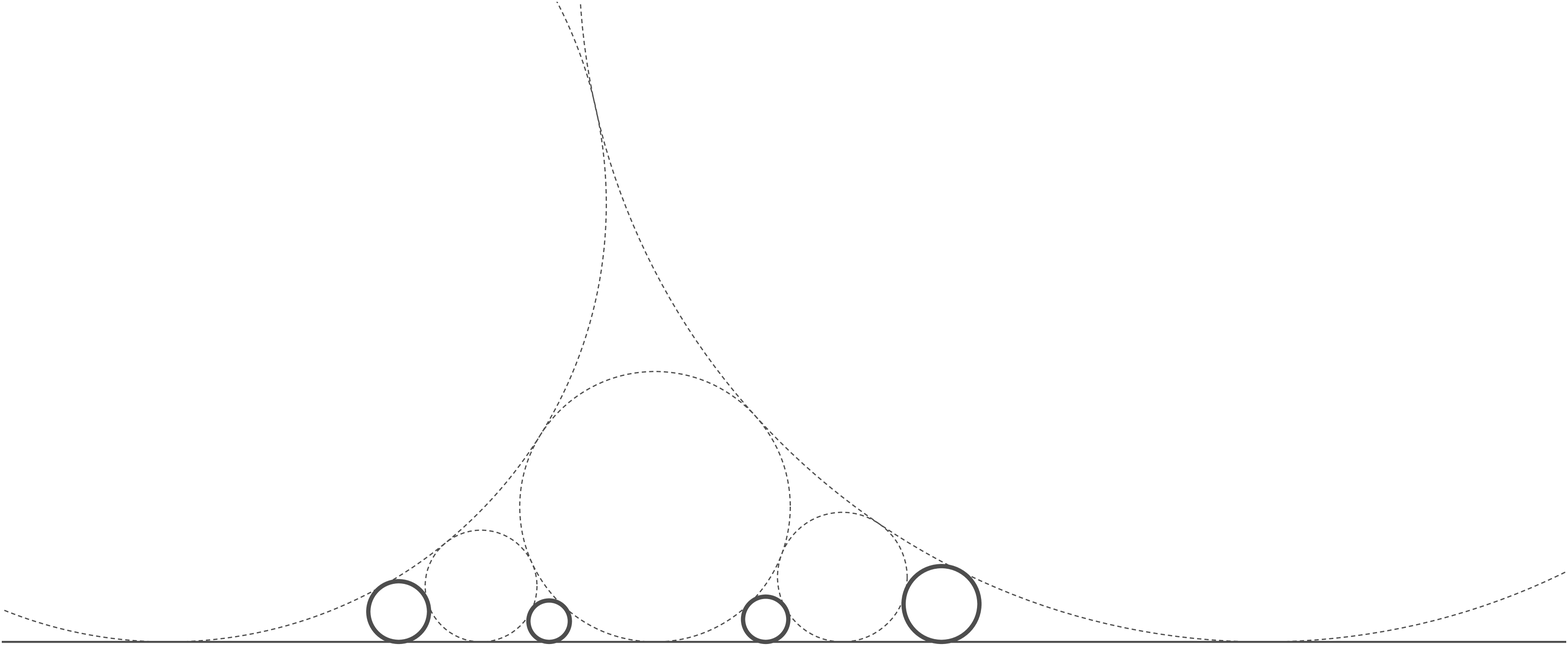}
		\caption{
			An instance of the graph $\bs G^{(2)}$ and the circle packing $\bs P^{(2)}$. For the vertex $v$, its parent $p(v)$ is shown, $p'(v)=u_l$ and its two offsprings are $v_1$ and $v_2$. Also, $\bs G'_2=\{v_1,v_2,v_3,v_4\}$. The circles in the circle packing $\bs P'_2$ are shown in bold.
			%
			%Later: $\bs G^{(3)}$, $\bs G'_3$ and a realization of $\bs P^{(3)}$.
		}
		\label{fig:pscp3}
	\end{figure}
	
	Let $n$ be fixed.
	Note that the radii in the restriction of $\bs P$ to $\bs G^{(n)}$ are uniquely determined by the radii of the two top vertices $u_r:=p^{n+1}(0)$ and $u_l:=p'(p^{n}(0))$. This is because the circle corresponding to $v$ should be tangent to the common line and to the circles of $p(v)$ and $p'(v)$. % (assuming it is in the bounded region between them). 
	A straightforward calculation shows that if $\bs s(v):=1/\sqrt{\bs r(v)}$, where $\bs r(v)$ is the radius, then
	\begin{equation*}
		\bs s(v)=\bs s(p(v))+ \bs s(p'(v)).
	\end{equation*}
	If $p'(v)$ is at the left of $p(v)$, consider the column vector $\bs X(v):=[\bs s(p'(v)), \bs s(p(v))]^t$. Otherwise, let $\bs X(v):=[\bs s(p(v)), \bs s(p'(v))]^t$.
	%Let $\bs X(v):=[\bs s(p'(v)), \bs s(p(v))]^t$ be a column vector. 
	The structure of the graph (see Figure~\ref{fig:pscp3}) and the above equation imply that the vectors corresponding to the offsprings of $v$ are $B \bs X(v)$ and $B^t \bs X(v)$, where $B:=\left[\begin{array}{cc} 1 & 1\\ 0& 1	\end{array} \right]$. Let $\bs o'$ be a uniform random vertex of $\bs G'_n$. It follows that conditioned on $\bs X(p^n(0))=x$, $\bs X(\bs o')$ has the same distribution as the product $B_n\cdots B_1 x$, where each $B_i$ is either $B$ or $B^t$, each with probability $1/2$, independently. Since $\bs s(\bs o')$ is the sum of the coordinates of $\bs X(\bs o')$, it has the same distribution as $[1\; 1] B_n\cdots B_1x$. Let $\bs Z:=[1\; 1] \bs X(p^n(0))= [1\; 1] x$. Therefore, by the central limit theorem for products of random matrices (Theorem~3 of~\cite{He97}), there exists $\gamma\in\mathbb R$ and $\sigma\geq 0$ such that the distribution of the random variable
	\[
		\frac 1{\sqrt n} \left[\log \frac {\bs s(\bs o')}{\bs Z}-n\gamma \right],
	\]
	conditioned on $\bs X(p^n(0))=x$, is approximately normal with zero mean and variance $\sigma^2$ (when $n$ is large). In addition, the (Prokhorov) distance between the distribution of this random variable and the normal distribution has a uniform bound regardless of $x\in (\mathbb R^+)^2\setminus 0$ (and the bound tends to zero as $n\to \infty$). Moreover,
	%converges weakly to a normal random variable with zero mean and variance $\sigma^2$ (as $n\to\infty$). In addition, 
	Corollary~3 and Theorem~5 of~\cite{He97} imply that $\sigma>0$. 
	Let $q_{\bs s}(\epsilon)$ be the $\epsilon$-quantile of $\{\bs s(v): v\in\bs G'_n \}$. Since $\bs o'$ is uniform in $\bs G'_n$ (given $\bs G'_n$ and $\bs P'_n$), the above convergence implies that 
	\[
		\frac 1{\sqrt n} \left[\log \frac {q_{\bs s}(\epsilon)}{\bs Z}-n\gamma \right] \to \sigma q_{\epsilon},
	\]
	where $q_{\epsilon}$ is the $\epsilon$-quantile of the standard normal distribution (note that the left hand side is a random variable depending on $\bs P'_n$, but the above argument shows that the convergence holds for any realization of $\bs P'_n$; i.e., the convergence is in $L^{\infty}$). 
	By $\bs s=1/\sqrt{\bs r}$, one has $q_{\bs s}(\epsilon)\approx 1/\sqrt{q(\bs P'_n,1-\epsilon)}$. So
	%Since $\bs o'$ is uniform in $\bs G'_n$ (given $\bs G'_n$ and $\bs P'_n$), one obtains that for every $\epsilon\in(0,1)$, \mar{\ali{Later: This is for $\bs s$, not for $\bs r$. Convert it.}}
	\[
	\frac 1{\sqrt n} \left[-\frac 12 \log \frac {q(\bs P'_n, \epsilon)}{\bs Z^{-2	}}-n\gamma \right] \to \sigma q_{1-\epsilon}.
	\]
	Subtraction gives
	\[
		\frac 1{\sqrt n} \log \frac {q(\bs P'_n, 1-\epsilon)}{q(\bs P'_n,\epsilon)} \to 2\sigma \left(q_{1-\epsilon}-q_{\epsilon}\right).
	\]
	This implies that~\eqref{eq:quantiles} does not hold for $\bs c'_n$. Hence, Theorem~\ref{thm:quantiles} implies that $[\bs G',0]$ does not have any point-stationary CP such that its CP-cocycle is $\bs c'$. So, by the first part of the proof, the claim is proved.
\end{proof}

\section*{Acknowledgements}
This research was in part supported by a grant from IPM (No. 98490118). The author thanks Mir-Omid Haji-Mirsadeghi for the verification of the proof of Theorem~\ref{prop:pscp}. We also thank Meysam Nassiri and Hesameddin Rajabzadeh for fruitful discussions on the circle packing used in Proposition~\ref{prop:scpStronger}.

\bibliography{bib} 

\begin{thebibliography}{10}

\bibitem{processes}
D.~Aldous and R.~Lyons.
\newblock Processes on unimodular random networks.
\newblock {\em Electron. J. Probab.}, 12:no. 54, 1454--1508, 2007.

\bibitem{fringe}
David Aldous.
\newblock Asymptotic fringe distributions for general families of random trees.
\newblock {\em Ann. Appl. Probab.}, 1(2):228--266, 1991.

\bibitem{AnHuNaRa16}
O.~Angel, T.~Hutchcroft, A.~Nachmias, and G.~Ray.
\newblock Unimodular hyperbolic triangulations: circle packing and random walk.
\newblock {\em Invent. Math.}, 206(1):229--268, 2016.

\bibitem{AnHuNaRa18}
O.~Angel, T.~Hutchcroft, A.~Nachmias, and G.~Ray.
\newblock Hyperbolic and parabolic unimodular random maps.
\newblock {\em Geom. Funct. Anal.}, 28(4):879--942, 2018.

\bibitem{AnSc03}
O.~Angel and O.~Schramm.
\newblock Uniform infinite planar triangulations.
\newblock {\em Comm. Math. Phys.}, 241(2-3):191--213, 2003.

\bibitem{II}
F.~Baccelli, M.-O. Haji-Mirsadeghi, and A.~Khezeli.
\newblock On the dimension of unimodular discrete spaces, part {II}: Relations
  with growth rate.
\newblock {\em preprint}.

\bibitem{I}
F.~Baccelli, M.-O. Haji-Mirsadeghi, and A.~Khezeli.
\newblock Unimodular {H}ausdorff and {M}inkowski dimensions.
\newblock {\em preprint}.

\bibitem{eft}
F.~Baccelli, M.~O. Haji-Mirsadeghi, and A.~Khezeli.
\newblock Eternal family trees and dynamics on unimodular random graphs.
\newblock In {\em Unimodularity in randomly generated graphs}, volume 719 of
  {\em Contemp. Math.}, pages 85--127. Amer. Math. Soc., Providence, RI, 2018.

\bibitem{BeTi19}
I.~Benjamini and A.~Timar.
\newblock Invariant embeddings of unimodular random planar graphs.
\newblock {\em arXiv preprint arXiv:1910.01614}, 2019.

\bibitem{GuNa13}
O.~Gurel-Gurevich and A.~Nachmias.
\newblock Recurrence of planar graph limits.
\newblock {\em Ann. of Math. (2)}, 177(2):761--781, 2013.

\bibitem{HeSc93}
Z.~He and O.~Schramm.
\newblock Fixed points, {K}oebe uniformization and circle packings.
\newblock {\em Ann. of Math. (2)}, 137(2):369--406, 1993.

\bibitem{HeSc95}
Z.~He and O.~Schramm.
\newblock Hyperbolic and parabolic packings.
\newblock {\em Discrete Comput. Geom.}, 14(2):123--149, 1995.

\bibitem{He97}
H.~Hennion.
\newblock Limit theorems for products of positive random matrices.
\newblock {\em Ann. Probab.}, 25(4):1545--1587, 1997.

\bibitem{shift-coupling}
A.~Khezeli.
\newblock Shift-coupling of random rooted graphs and networks.
\newblock {\em To appear in the special issue of Contemporary Mathematics on
  Unimodularity in Randomly Generated Graphs}, 2018.

\bibitem{bookKo36}
P.~Koebe.
\newblock {\em Kontaktprobleme der konformen Abbildung}.
\newblock Hirzel, 1936.

\bibitem{ThLa09}
G.~Last and H.~Thorisson.
\newblock Invariant transports of stationary random measures and
  mass-stationarity.
\newblock {\em Ann. Probab.}, 37(2):790--813, 2009.

\bibitem{bookTh79}
William~P Thurston.
\newblock {\em The geometry and topology of three-manifolds}.
\newblock Princeton University Princeton, NJ, 1979.

\end{thebibliography}
\bibliographystyle{plain}

\end{document}